\documentclass[twoside,11pt]{article}
\usepackage[nohyperref]{jmlr2e}

\usepackage{microtype}

\usepackage{url}
\usepackage[CJKbookmarks=true,
            bookmarksnumbered=true,
            bookmarksopen=true,
            colorlinks,
            linktocpage,
            linkcolor={blue!70!gray}, citecolor={blue!70!gray}, urlcolor={blue!70!gray}]{hyperref}

\usepackage{amssymb,amsmath,amsfonts,bm}
\usepackage{autobreak}
\usepackage{pifont}

\usepackage{algorithm}
\usepackage{algorithmic}
\usepackage{booktabs}
\usepackage{multirow,bigstrut,threeparttable}
\usepackage{graphicx}
\usepackage{subfigure}
\usepackage{epstopdf}
\usepackage{tcolorbox}
\usepackage{soul}
\usepackage{float}
\usepackage{tikz}
\usepackage{tikz-3dplot}
\usepackage{blindtext}

\usepackage{multirow}
\usepackage{adjustbox}
\usepackage[graphicx]{realboxes}
\usepackage{rotating}
\usepackage{array}
\usepackage{makecell}
\usepackage{float}
\usepackage{booktabs}
\usepackage[section]{placeins}

\newcommand{\Appendix}[1]{the full version for}

\renewcommand{\L}{\mathbf{L}}

\renewcommand{\L}{\textsf{L}}

 \usepackage{lastpage}

\ShortHeadings{Unsupervised Feature Selection via NOCRM}{Li, Sun and Zhang}
\firstpageno{1}

\begin{document}

\title{Unsupervised Feature Selection via Nonnegative Orthogonal Constrained Regularized Minimization}

\author{\name Yan Li \email li-yan20@mails.tsinghua.edu.cn \\
\addr Department of Mathematical Sciences\\
Tsinghua University\\
Beijing 100084, China
\AND
\name Defeng Sun \email defeng.sun@polyu.edu.hk\\
\addr Department of Applied Mathematics\\
The Hong Kong Polytechnic University\\
Hung Hom, Kowloon, Hong Kong
\AND
\name  Liping Zhang\protect\footnotemark[1] \email lipingzhang@tsinghua.edu.cn\\
\addr Department of Mathematical Sciences\\
Tsinghua University\\
 Beijing 100084, China}

\renewcommand{\thefootnote}{\fnsymbol{footnote}}
\footnotetext[1]{Liping Zhang is the corresponding author (lipingzhang@tsinghua.edu.cn).}

\editor{}
\maketitle

\begin{abstract}%
Unsupervised feature selection has drawn wide attention in the era of big data since it is a primary technique for dimensionality reduction. However, many existing unsupervised feature selection models and solution methods were presented for the purpose of application, and lack of theoretical support, e.g., without convergence analysis. In this paper, we first establish a novel unsupervised feature selection model based on regularized minimization with nonnegative orthogonal constraints, which has advantages of embedding feature selection into the nonnegative spectral clustering and preventing overfitting. An effective inexact augmented Lagrangian multiplier method is proposed to solve our model, which adopts the proximal alternating minimization method to solve subproblem at each iteration. We show that the sequence generated by our method globally converges to a Karush-Kuhn-Tucker point of our model. Extensive numerical experiments on popular datasets demonstrate the stability and robustness of our method. Moreover, comparison results of algorithm performance show that our method outperforms some existing state-of-the-art methods. 
\end{abstract}

\vspace{1mm}
\section{Introduction}\label{sec:intro}
 Due to large amounts of data produced by rapid development of technology, processing high-dimensional data is one of the most challenging problems in many fields, such as action recognition~\citep{klaser2011action}, image classification~\citep{gui2014group}, computational biology~\citep{chen2020feature}. Generally, not all the features are equally important for the data with high-dimensional features. There are some redundant, irrelevant and noisy features, which not only increase computational cost and storage burden, but also reduce the performance of learning tasks. Dimensionality reduction methods can be roughly divided into two types: feature extraction~\citep{lee1999learning, charte2021reducing, lian2018hierarchical} and feature selection~\citep{kittler1986feature,9580680,roffo2020infinite,yu2019multi}. They project the high-dimensional feature space to a low-dimensional space to squeeze features. The low-dimensional space generated by the former is usually composed of linear or nonlinear combinations of original features, but irrelevant, redundant and even noisy features are involved in the process of reducing dimension, which may affect the subsequent learning tasks to some extent. However, the latter evaluates each dimension feature of high dimensional data and directly select the optimal feature subset from the original high-dimensional feature set by using certain criteria to achieve compact and accurate data representation~\citep{liu2004selective, molina2002feature}. Compared with the former, the latter has better interpretability. Feature selection maintains the semantic information of the original features and it aims to select valuable and discriminative feature subsets from the original high-dimensional feature set, while feature extraction changes the original meanings of the feature and the new features usually lose the physical meanings of the original features. Therefore, feature selection enjoys tremendous popularity in a wide range of applications from data mining to machine learning. Many feature selection methods~\citep{nie2016unsupervised,hou2013joint,nie2019structured} are proposed to better explore the properties of high-dimensional data.

According to whether the class label information is available or not, feature selection methods can be roughly grouped into two categories, i.e., supervised feature selection, and unsupervised feature selection~\citep{dash1997dimensionality,he2005laplacian}. Benefiting from the sample-wise annotations, supervised feature selection algorithms, e.g., Fisher score~\citep{duda2001stork}, robust regression~\citep{nie2010efficient}, minimum redundancy maximum relevance ~\citep{peng2005feature} and trace radio~\citep{nie2008trace}, are able to select discriminative features and achieve superior classification accuracy and reliability.
With the fact that the labeled data is often inadequate or completely unobtainable in many practical applications, traditional supervised feature selection methods cannot deal with such problems. In addition, annotating the unlabeled data requires an excessive cost in human resources and is time-consuming. Therefore, for the high-dimensional data with missing labels, it is an  effective means to solve above mentioned problems by using unsupervised approaches to reduce the feature dimension. Compared to supervised feature selection, unsupervised feature selection is a more challenging task since the label information of the training data is unavailable~\citep{he2005laplacian}. Many studies have been conducted on unsupervised feature selection methods, such as spectral analysis ~\citep{zhao2007spectral,cai2010unsupervised,li2012unsupervised}, matrix factorization  ~\citep{wang2015embedded,qian2013robust}, dictionary learning~\citep{zhu2016coupled} and so on.

Unsupervised feature selection methods~\citep{nie2019structured,chen2022fast} generally select features according to the intrinsic structural characteristics of data and have achieved pretty good performance, which can alleviate the undesirable influence of noise and redundant features in the original data. For example, MaxVar~\citep{krzanowski1987selection} is a statistical method, which selects features corresponding to the maximum variance.  Laplacian Score~\citep{he2005laplacian} is a similarity preserving method, which evaluates the importance of a feature by its power of locality preservation;  SPEC~\citep{zhao2007spectral} selects features using spectral regression. RSR~\citep{zhu2015unsupervised} is a data reconstruction method, which uses the $l_{2,1}$-norm to measure the fitting error and to promote sparsity; CPFS~\citep{masaeli2010convex} relaxes the feature selection problem into a continuous convex optimization problem; REFS~\citep{li2017reconstruction}  embeds the reconstruction function learning process to feature selection. MCFS~\citep{cai2010unsupervised}  selects features based on spectral analysis and sparse regression problem, UDFS~\citep{yang2011l2} which selects features by preserving the structure based on discriminative information; UDPFS~\citep{wang2020unsupervised} introduces fuzziness into sub-space learning to learn a discriminative projection for feature selection; NDFS~\citep{li2012unsupervised} selects features by leveraging a joint framework of nonnegative spectral analysis and $l_{2,1}$-norm regularization.
However, numerical algorithms proposed in  these unsupervised feature selection methods are often without global convergence analysis and then lack of theoretical support~\citep{shi2016primer}.
Furthermore, these methods may be greatly affected by disturbance and do not have good performance, and then they  may not have good stability and strong robustness.

 Motivated by this, we establish a novel unsupervised feature selection model based on regularized minimization with nonnegative orthogonal constraints, which has two advantages of embedding feature selection into the nonnegative spectral clustering and preventing overfitting. In our model, the $l_{2,1}$-regularized term will enable the subproblem from our proposed algorithm has closed-form optimal solution, and the Frobenius-norm regularization will explicitly control the overfitting, which is the main difference from NDFS~\citep{li2012unsupervised}. And the nonnegative orthogonal constraints can embed feature selection into the nonnegative spectral clustering. However, it is hard to handle the orthogonal constraints in general~\citep{absil2009optimization}. Some existing popular solution methods, such as the multiplicative update method~\citep{ding2006orthogonal,yoo2008orthogonal} and the greedy orthogonal pivoting algorithm~\citep{zhang2019greedy}, require the objective function to be differentiable and have the special formulation. So, they are not applicable to our model. This prompts us to design an effective solution method for our model. Based on the algorithmic frameworks of the augmented Lagrangian method~\citep{andreani2008augmented} and the proximal alternating minimization~\citep{attouch2013convergence}, we propose an effective inexact augmented Lagrangian multiplier (ALM) method to solve our model, which uses the proximal alternating minimization (PAM) method to solve subproblems at each iteration. We show that the sequence generated by our ALM method globally converges to a Karush-Kuhn-Tucker point of our model. Numerical experiments on popular datasets demonstrate the stability and robustness of our method. Moreover, comparison results of algorithm performance show that our method outperforms some existing state-of-the-art methods.

The main contribution of this paper is summarized as follows:
\begin{itemize}
    \item  We establish a novel $l_{2,1}$-regularized optimization model with nonnegative orthogonal constraints for unsupervised feature selection, which has two advantages of embedding feature selection into the nonnegative spectral clustering and preventing overfitting. Specifically, we use the  spectral clustering technique to learn pseudo class labels, and then select features which are most discriminative to pseudo class labels.

    \item We propose an effective  inexact ALM method to solve our model. At each iteration, we use the PAM method to solve subproblems, which has the advantage of making each subproblem have a closed form solution. This helps us to show that the sequence generated by our ALM method globally converges to a Karush-Kuhn-Tucker point of our model without any further assumption.
 Numerical results on popular datasets are reported to show the efficiency, stability and robustness of our method.
\end{itemize}

The rest of this paper is organized as follows. In Section 2, some preliminaries for nonsmooth  optimization are collected. In Section 3, we establish a novel model for unsupervised feature selection. In Section 4, an inexact ALM method is proposed to solve our model, and its convergence analysis is also given.  Numerical experiments and concluding remarks are given in the last two sections.

\section{Preliminaries}
In this section, we recall some preliminaries on nonsmooth optimization and give some notations. Throughout this paper, matrices are written as capital letters (e.g., $A,B,\cdots$) and vectors are denoted as boldface lowercase letters (e.g., ${\bf x},{\bf y},\cdots$). For any positive integer $n$, denote $[n]=\{1,2,\ldots,n\}$. Given a matrix $Y=(Y_{i,j})\in\mathbb{R}^{n\times m}$,  its maximum (elementwise) norm is denoted by $$\Vert Y\Vert_{\infty}:=\max\{\vert Y_{i,j}\vert:\ i\in[n],\,j\in[m]\}.$$
The Frobenius norm of $Y$ is denoted by $$\Vert Y\Vert_F:=\sqrt{\sum_{i=1}^{n}\sum_{j=1}^{m}Y_{i,j}^2},$$
and its $l_{2,1}$-norm is defined as $$\Vert Y\Vert_{2,1}:=\sum_{i=1}^{n}\sqrt{\sum_{j=1}^{m}Y_{i,j}^2}=\sum_{i=1}^{n}\Vert Y_i\Vert_2,$$
where $Y_i$ is the $i$-th row of $Y$ and $\Vert\cdot\Vert_2$ is Euclidean norm.
Let $\text{Vec}(Y)$ be an $mn\times 1$ vector with $\text{Vec}(Y):=[Y_1^\intercal, Y_2^\intercal,\cdots, Y_m^\intercal]^\intercal$. For any $\textbf{v}\in\mathbb{R}^n$, let $[\textbf{v}]_i$ denote its $i$th component, and let $\text{diag}(\textbf{v})\in\mathbb{R}^{n\times n}$ denote the diagonal matrix with diagonal entries $\{[\textbf{v}]_i\}_{i=1}^n$. Given a square matrix $Y$, $Y\succ 0$ denotes that $Y$ is a positive definite matrix and the trace of $Y$, i.e., the sum of the diagonal elements of $Y$, is denoted by $\text{Tr}(Y)$. $\textbf{E}$ is a matrix whose elements are all $1$. $\textbf{O}$ is a matrix whose elements are all $0$. $\textbf{O}\leq X\leq \textbf{E}$ denotes that each element of X satisfies $0\leq X_{i,j}\leq 1$. $\textbf{0}\leq \textbf{v} \leq\textbf{1}$ denotes that each element of $\textbf{v}$ satisfies $0\leq [\textbf{v}]_i\leq 1$. Given a set $\Omega$, $\Pi_\Omega Y$ denotes the projection of $Y$ on $\Omega$. For an index sequence $\mathcal{K}=\{k_0, k_1, k_2,\ldots\}$ that satisfies $k_{j+1}> k_j$ for any $j\geq 0$, we denote $\lim_{k\in\mathcal{K}}x_k:=\lim_{j\rightarrow\infty}x_{k_j}$. For any set $S$, its indicator function is defined by
\begin{equation}
    \delta_S(X)=
    \left\{
    \begin{array}{cc}
         0,& \text{if}\;X\in S, \\
         +\infty,&\text{otherwise}.
    \end{array}
    \right.
\end{equation}

Let us recall some definitions of sub-differential calculus~\citep[see, e.g.,][]{rockafellar2009variational}.
\begin{definition}
Let $C\subseteq\mathbb{R}^n$ and $\overline{\textbf{x}}\in C$. A vector $\textbf{v}$ is normal to $C$ at $\overline{\textbf{x}}$ in the regular sense, or a regular normal, written $\textbf{v}\in\hat{N}_{C}(\overline{\textbf{x}})$, if
$$\langle\textbf{v},\textbf{x}-\overline{\textbf{x}}\rangle\leq \textbf{o}(\Vert\textbf{x}-\overline{\textbf{x}}\Vert)\ \text{for}\ \textbf{x}\in C.$$
A vector is normal to $C$ at $\overline{\textbf{x}}$ in the general sense, written $\textbf{v}\in N_{C}(\overline{\textbf{x}})$, if there exists sequence $\{\textbf{x}_k\}_k\subset C, \{\textbf{v}_k\}_k$ such that $\textbf{x}_k\rightarrow\overline{\textbf{x}}$ and $\textbf{v}_k\rightarrow\textbf{v}$ with $\textbf{v}_k\in\hat{N}_{C}(\textbf{x}_k)$. The cone $N_{C}(\overline{\textbf{x}})$ is called the normal cone to $C$ at $\overline{\textbf{x}}$.
\end{definition}
\begin{definition}
Let $f\!:\!\mathbb{R}^n\rightarrow \mathbb{R}\cup\{+\infty\}$ be a proper lower semicontinuous function.
\begin{itemize}
  \item[1)] The domain of $f$ is defined and denoted by $\text{dom}\;f\!:=\!\{\textbf{x}\in\mathbb{R}^n:\!f(\textbf{x})\textless +\infty\}$.
  \item[2)] For each $\textbf{x}\in \text{dom}\;f$, the vector $\textbf{x}^{*}\in\mathbb{R}^n$ is said to be a regular subgradient of $f$ at $\textbf{x}$, written $\textbf{x}^{*}\in\hat{\partial}f(\textbf{x})$, if
  $f(\textbf{y})\geq f(\textbf{x})+\langle\textbf{x}^*,\textbf{y}-\textbf{x}\rangle+\textbf{o}(\Vert\textbf{x}-\textbf{y}\Vert)$.
  \item[3)] The vector $\textbf{x}^{*}\in\mathbb{R}^n$ is said to be a (limiting) subgradient of $f$ at $\textbf{x}\in\text{dom}\;f$, written $\textbf{x}^{*}\in\partial f(\textbf{x})$, if there exists $\{\textbf{x}_n\}_n$, $\{\textbf{x}_n^*\}_n$ such that $\textbf{x}_n\rightarrow\textbf{x}, f(\textbf{x}_n)\rightarrow f(x)\;\text{and}\;\textbf{x}_n^*\in\hat{\partial}f(\textbf{x}_n)\;\text{with}\;\textbf{x}_n^*\rightarrow\textbf{x}^*$.
  \item[4)] For each $\textbf{x}\in \text{dom}\;f$, $\textbf{x}$ is called (limiting)-critical if $\textbf{0}\in\partial f(\textbf{x})$.
\end{itemize}
\begin{remark}[Closedness of $\partial f$]\label{rma1}
Let $(\textbf{x}_k,\textbf{x}_k^*)\in \text{Graph}\;\partial f$ be a sequence that converges to $(\textbf{x},\textbf{x}^*)$. By the definition of $\partial f(\textbf{x})$, if $f(\textbf{x}_k)$ converges to $f(\textbf{x})$ then $(\textbf{x},\textbf{x}^*)\in \text{Graph}\;\partial f$.
\end{remark}
\begin{remark}\citep[Example~6.7]{rockafellar2009variational}\label{rma2}
   Let $S$ be a closed nonempty subset of $\mathbb{R}^n$, then $$\partial\delta_{S}(\overline{\textbf{x}})=N_S(\overline{\textbf{x}}),\ \overline{\textbf{x}}\in S.$$
   Furthermore, for a smooth mapping $G:\mathbb{R}^n\rightarrow\mathbb{R}^m$, i.e., $G(\textbf{x}):=(g_1(\textbf{x}),\cdots, g_m(\textbf{x}))^\intercal$, define
   $S=G^{-1}(\textbf{0})\subset\mathbb{R}^n$. Set $\nabla G(\textbf{x}):=[\frac{\partial g_j}{\partial\textbf{x}_i}(\textbf{x})]_{i,j=1}^{n,m}\in\mathbb{R}^{n\times m}$.
   If $\nabla G(\overline{\textbf{x}})$ has full rank $m$ at a point $\overline{\textbf{x}}\in S$, with $G(\overline{\textbf{x}})=\textbf{0}$,
   then its normal cone to $S$ can be explicitly written as
   \begin{equation*}
   N_S(\overline{\textbf{x}})=\{\nabla G(\overline{\textbf{x}})\textbf{y}\;\vert\;\textbf{y}\in\mathbb{R}^m\}.
   \end{equation*}
 \end{remark}
 \end{definition}

\section{A New Unsupervised Feature Selection Model}
Let $X=[\textbf{x}_1,\textbf{x}_2,\cdots,\textbf{x}_n]\in\mathbb{R}^{d\times n}$ be the data matrix with each column $\textbf{x}_i\in\mathbb{R}^{d\times 1}$ being the $i$-th data point. Let $d$ and $n$ be the number of features and the number of sample, respectively. Suppose these $n$ samples are sampled from $c$ classes. Denote $F=[\textbf{f}_1,\cdots,\textbf{f}_n]^\intercal\in\{0, 1\}^{n\times c}$, where $\textbf{f}_i\in\{0,1\}^{c\times 1}$ is the cluster indicator vector for $\textbf{x}_i$. That is, the $j$-th element of $\textbf{f}_i$ is $1$, if $\textbf{x}_i$ is assigned to the $j$-th cluster, otherwise $0$.
Following the notation in \citet{yang2011nonnegative}, the scaled cluster indicator matrix $Y$ is defined as $$Y=[\textbf{y}_1, \textbf{y}_2,\cdots,\textbf{y}_n]^\intercal=F(F^\intercal F)^{-\frac{1}{2}},$$ where $y_i$ is the scaled cluster indicator of $\textbf{x}_i$. It turns out that $$Y^\intercal Y=(F^\intercal F)^{-\frac{1}{2}}F^{\intercal}F(F^\intercal F)^{-\frac{1}{2}}=I_c,$$ where $I_c\in\mathbb{R}^{c\times c}$ is an identity matrix.

At first, we use the clustering techniques to learn the scaled cluster indicators of data points, which can be regarded as pseudo class labels.  Given a set of data points $\textbf{x}_1,\textbf{x}_2,\cdots,\textbf{x}_n$ and some notion of similarity $s_{i,j}\geq 0$ between all pairs of data points $\textbf{x}_i$ and $\textbf{x}_j$, the intuitive goal of clustering is to divide the data points into several groups such that points in the same group are similar and points in different groups are dissimilar to each other. Spectral clustering is widely used in that it can effectively generate the pseudo labels from the graphs. In our method, we construct a $k$-nearest neighbors graph and choose the Gaussian kernel as the weight~\citep[see][]{cai2005document}. Specially, we define the affinity graph $S$ as follows:
\begin{equation*}
    S_{i,j}=
    \left\{
    \begin{array}{cc}
     \text{exp}(-\frac{\Vert\textbf{x}_i-\textbf{x}_j\Vert^2}{2\sigma^2}),&\textbf{x}_i\in\mathcal{N}_k(\textbf{x}_j)\;\text{or}\;\textbf{x}_j\in\mathcal{N}_k(\textbf{x}_i)  \\
       0,  &\text{otherwise},
    \end{array}
    \right.
\end{equation*}
where $\mathcal{N}_k(\textbf{x})$ is the set of $k$-nearest neighbors of $\textbf{x}$. The corresponding degree matrix can be constructed to $D$ with $D_{ii}=\sum_{j}S_{i,j}$, and Laplacian matrix $L$ of the normalized graph 
\citep[see][]{von2007tutorial} is calculated with $L=D^{-\frac{1}{2}}(D-S)D^{-\frac{1}{2}}$. Therefore, the local geometrical structure of data points can be obtained by:
\begin{equation}
    \min_{Y} \text{Tr}(Y^\intercal LY)\quad \text{s.t.}\quad Y=F(F^\intercal F)^{-\frac{1}{2}}.
\end{equation}
This is a discrete optimization problem as the entries of the feasible solution are only allowed to take two particular values, and of course it is a NP-hard problem. A well-known method is to discard the discreteness condition and relax the problem by allowing the entries of the matrix $Y$ to take arbitrary real values. Then, the relaxed problem becomes:
\begin{equation}\label{spec-clu}
    \min_{Y\in\mathbb{R}^{n\times c}} \text{Tr}(Y^\intercal LY)\quad\text{s.t.}\quad Y^\intercal Y=I_c.
\end{equation}

The next stage is to construct a sparse transformation $W$ on the data matrix $X$ by employing the scaled cluster indicator matrix $Y$, joined with two regularization terms. We can formulate it as:
\begin{equation}\label{regress}
    \min_{W\in\mathbb{R}^{d\times c}}\Vert Y-X^\intercal W\Vert_{2,1}+\beta\Vert W\Vert_{2,1}+\gamma\Vert W\Vert_F^2,
\end{equation}
where $W$ is a linear and low dimensional transformation matrix, and $\beta$ and $\gamma$ are the regularization parameters. In the objection function of the problem (\ref{regress}), the first term represents the linear transformation model to measure the association between the features and the pseudo class labels. The second term constructs the sparsity on the rows of the transformation matrix $W$, which is beneficial for selecting discriminative features. The third term is to avoid overfitting.

By integrating the spectral clustering (\ref{spec-clu}) and sparse regression (\ref{regress}) in a joint objective function, the model we proposed can be obtained as follows:
\begin{equation}\label{pro}
 \begin{aligned}
   &\min_{W,Y}\text{Tr}(Y^\intercal LY)+\alpha\Vert Y-X^\intercal W\Vert_{2,1}+\beta\Vert W\Vert_{2,1}+\gamma\Vert W\Vert_F^2\\
     &\text{s.t.}\quad Y^\intercal Y=I_c,\ Y\geq\textbf{O},
\end{aligned}
\end{equation}
where $\alpha$ is a tuning parameter.

\section{Algorithm Description of Our Inexact ALM Method}
In this section, we develop our inexact augmented Lagrangian method for solving problem (\ref{pro}), which is a nonconvex optimization with a nonsmooth objective function. By introducing auxiliary variables $U, V, \widehat{Y}, F$, the problem (\ref{pro}) can be transformed into the following equivalent:
\begin{equation}\label{pro2}
\begin{aligned}
 \mathop{\min}_{W,U,V,Y,F,\widehat{Y}}&\text{Tr}(Y^{\intercal}LY)+\alpha\Vert U\Vert_{2,1}+\beta\Vert V\Vert_{2,1}+\gamma\Vert W\Vert_F^2+\delta_{\mathbb{S}_1}(\widehat{Y})+\delta_{\mathbb{S}_2}(F)\\
  \text{s.t.}\qquad & \left\{
  \begin{array}{l}
   Y=F\\
   U=Y-X^{\intercal}W\\
   V=W\\
   Y=\widehat{Y}
  \end{array}
  \right.
\end{aligned}
\end{equation}
where $\mathbb{S}_1=\{\;\widehat{Y}\;\vert\;\widehat{Y}^{\intercal}\widehat{Y}=I_c\;\}$, $\mathbb{S}_2=\{\;F\;\vert\;\textbf{O}\leq F\leq\textbf{E}\;\}$.

Set $\lambda:=(\lambda_1,\lambda_2,\lambda_3,\lambda_4)\in\mathbb{R}^{n\times c}\times\mathbb{R}^{d\times c}\times\mathbb{R}^{n\times c}\times \mathbb{R}^{n\times c}$. The augmented Lagrangian function for (\ref{pro2}) is defined by
\begin{equation}
\begin{aligned}
  L(W,U,V,Y,F,\widehat{Y},\lambda;\rho):=&\text{Tr}(Y^{\intercal}LY)+\alpha\Vert U\Vert_{2,1}+\beta\Vert V\Vert_{2,1}+\gamma\Vert W\Vert_F^2+\delta_{\mathbb{S}_1}(\widehat{Y})\\&+\delta_{\mathbb{S}_2}(F)+\langle\lambda_1,Y-X^{\intercal}W-U\rangle+\langle\lambda_2,V-W\rangle\\&+\langle\lambda_3,Y-F\rangle+\langle\lambda_4,\widehat{Y}-Y\rangle+\frac{\rho}{2}(\Vert \widehat{Y}-Y\Vert_F^2+\Vert V-W\Vert_F^2\\&+\Vert Y-F\Vert_F^2+\Vert Y-X^{\intercal}W-U\Vert_F^2),
\end{aligned}
\end{equation}
where $\rho$ is a positive penalty parameter.

The ALM method can be used to alternately update the $(W, U, V, Y, F, \widehat{Y})$, the multiplier $\lambda$, and the penalty parameter $\rho$ to satisfy the accuracy condition (\ref{Critical}). We describe our inexact ALM method for solving (\ref{pro}) in details as follows.

\begin{algorithm}[ht]
\caption{Inexact ALM Method for (\ref{pro})}\label{algo1}
\begin{algorithmic}[0]
    \STATE {{\bf Input.} Data matrix $X\in\mathbb{R}^{d\times n}$. Given predefined parameters $\{\epsilon_k\}_{k\in\mathbb{N}}$, $\rho^1$, $\tau$, $r$, $\overline{\lambda}_{N, min}$,  $\overline{\lambda}_{N, max}$ $(N=1,2,3,4)$, and $\overline{\lambda}^1:=(\overline{\lambda}^1_1, \overline{\lambda}^1_2, \overline{\lambda}^1_3, \overline{\lambda}^1_4)$ that satisfy the condition in Remark \ref{rem1}, for $k=1,2,\ldots,$}\\
    \STATE {{\bf Output.} Sort all the $d$ features according to $\Vert W_i\Vert_2\, (i\in[d])$ and select the top $q$ ranked features.}\\
	 \textbf{Step 1:} Compute the subproblem
	\begin{equation}\label{sublagpro}
	 (W^k,U^k,V^k,Y^k,F^k,\widehat{Y}^k)\approx\mathop{\arg\min}_{W,U,V,Y,F,\widehat{Y}} L(W,U,V,Y,F,\widehat{Y},\overline{\lambda}^k;\rho^k)
	\end{equation}
	such that
	\begin{equation}\label{constrain}
	 \textbf{O}\leq F^k\leq E, (\widehat{Y}^{k})^{\intercal}\widehat{Y}^k=I_c,
	\end{equation}
	and there exists
	$\Theta^k\in \partial L(W^k,U^k,V^k,Y^k,F^k,\widehat{Y}^k,\overline{\lambda}^k;\rho^k)$
 satisfying
	\begin{equation}\label{Critical}
	\Vert\Theta^k\Vert_{\infty}\leq\epsilon_k.
	\end{equation}
	 \textbf{Step 2:} Update the multiplier as:
	 \begin{equation*}
	 \begin{aligned}
	 &\lambda_1^{k+1}=\overline{\lambda}_1^k+\rho^k(Y^k-X^\intercal W^k-U^k)& \nonumber\\
		&\lambda_2^{k+1}=\overline{\lambda}_2^k+\rho^k(V^k-W^k)& \nonumber\\
		&\lambda_3^{k+1}=\overline{\lambda}_3^k+\rho^k(Y^k-F^k)& \nonumber\\
		&\lambda_4^{k+1}=\overline{\lambda}_4^k+\rho^k(\widehat{Y}^k-Y^k)&\nonumber
	 \end{aligned}
   	\end{equation*}
	where $\overline{\lambda}_N^{k+1}=\Pi_{\Omega}\lambda_N^{k+1}$ and    $\Omega=\{\lambda_N:\overline{\lambda}_{N,min}\leq\lambda_N\leq\overline{\lambda}_{N,max}\}$, $N=1,2,3,4$.\\
	
    \textbf{Step 3:} Update the penalty parameter:
	\begin{equation}\label{rho}
		\rho^{k+1}=\left\{
		\begin{aligned}
			\rho^k\quad &\text{if}\ \Vert R_i^k\Vert_{\infty}\leq \tau\Vert R_i^{k-1}\Vert_{\infty}\ (i=1,2,3,4)\\
			r\rho^k\quad&\text{otherwise},
		\end{aligned}
	\right.
	\end{equation}
	where
	 $R_1^k=Y^k-X^\intercal W^k-U^k,\ R_2^k= V^k-W^k,
	         R_3^k=Y^k-F^k,\ R_4^k=\widehat{Y}^k-Y^k$.
\end{algorithmic}
\end{algorithm}
\begin{remark}\label{rem1}
	Set the parameters in Algorithm \ref{algo1} as follows: $\tau\in[0,1)$; $\rho^1>0$; $r>1$;  the sequence of positive tolerance parameters $\{\epsilon_k\}_{k\in\mathbb{N}}$ is chosen such that $\lim_{k \to +\infty}\epsilon_k=0$. The parameters $\overline{\lambda}_1^1, \overline{\lambda}_2^1, \overline{\lambda}_3^1, \overline{\lambda}_4^1, \overline{\lambda}_{N, min}, \overline{\lambda}_{N, max}$ are finite-valued matrices satisfying
	\begin{equation*}
	-\infty<[\overline{\lambda}_{N, min}]_{i,j}<[\overline{\lambda}_{N, max}]_{i,j}<+\infty\ \forall i,j,\ \; N=1,2,3,4.
	\end{equation*}
\end{remark}

In Algorithm \ref{algo1}, the most important is how to solve (\ref{sublagpro}-\ref{Critical}). That is, given the current iterate $(W^k, U^k, V^k$, $Y^k, F^k, \widehat{Y}^k)$, how to generate the next iterate $(W^{k+1}$, $U^{k+1}$, $V^{k+1}$, $Y^{k+1}$, $F^{k+1}$, $\widehat{Y}^{k+1})$ is very critical. We propose a PAM method to solve (\ref{sublagpro}-\ref{Critical}) and show that there exists a solution for (\ref{sublagpro}-\ref{Critical}) and such a solution can be efficiently computed  as $\epsilon_k\!\downarrow 0$, i.e., Step 1 in Algorithm \ref{algo1} is well defined. We will establish the PAM method and its convergence in the following two subsections.

\subsection{PAM for Augmented Lagrangian Subproblem}
In this subsection, we present more details on implementing Algorithm \ref{algo1} and construct a PAM method to solve the augmented Lagrangian subproblem with arbitrarily given accuracy.

It can be seen that the constraint (\ref{Critical}) is an $\epsilon^k$-perturbation of the critical point property
\begin{equation}\label{criL}
    0\in\partial L(W^k,U^k,V^k,Y^k,F^k,\widehat{Y}^k,\overline{\lambda}^k;\rho^k).
\end{equation}
In fact, the algorithm we proposed to deal with (\ref{criL}) is a regularized proximal six-block Gauss-Seidel method. At the $k$th outer iteration,
the problem (\ref{criL}) can be solved with arbitrary accuracy using the following alternating minimizing procedure:
\begin{itemize}
    \item[(a)] Update $W^{k,j}$:
    \begin{small}
    \begin{equation}\label{update_W}
    W^{k,j}\in \mathop{\arg\min}\;\{L(W, U^{k,j-1}, V^{k,j-1}, Y^{k,j-1}, F^{k,j-1}, \widehat{Y}^{k,j-1}, \overline{\lambda}^{k};\rho^k)+
     \frac{C_1^{k,j-1}}{2}\Vert W-W^{k,j-1}\Vert_F^2\},
    \end{equation}
    \end{small}
    \item[(b)] Update $U^{k,j}$:
     \begin{small}
    \begin{equation}\label{update_U}
    U^{k,j}\in \mathop{\arg\min}\;\{L(W^{k,j},U,V^{k,j-1},Y^{k,j-1},F^{k,j-1},\widehat{Y}^{k,j-1},\overline{\lambda}^{k};\rho^k)+
    \frac{C_2^{k,j-1}}{2}\Vert U-U^{k,j-1}\Vert_F^2\},
    \end{equation}
     \end{small}
    \item[(c)] Update $V^{k,j}$:
     \begin{small}
    \begin{equation}\label{update_V}
    V^{k,j}\in \mathop{\arg\min}\;\{L(W^{k,j},U^{k,j},V,Y^{k,j-1},F^{k,j-1},\widehat{Y}^{k,j-1},\overline{\lambda}^{k};\rho^k)+
   \frac{C_3^{k,j-1}}{2}\Vert V-V^{k,j-1}\Vert_F^2\},
    \end{equation}
     \end{small}
      \item[(d)] Update $Y^{k,j}$:
      \begin{small}
    \begin{equation}\label{update_Y}
    Y^{k,j}\in \mathop{\arg\min}\;\{L(W^{k,j},U^{k,j},V^{k,j},Y,F^{k,j-1},\widehat{Y}^{k,j-1},\overline{\lambda}^{k};\rho^k)+
   \frac{C_4^{k,j-1}}{2}\Vert Y-Y^{k,j-1}\Vert_F^2\},
    \end{equation}
     \end{small}
      \item[(e)] Update $F^{k,j}$:
       \begin{small}
    \begin{equation}\label{update_F}
    F^{k,j}\in \mathop{\arg\min}\;\{L(W^{k,j},U^{k,j},V^{k,j},Y^{k,j},F,\widehat{Y}^{k,j-1},\overline{\lambda}^{k};\rho^k)+
    \frac{C_5^{k,j-1}}{2}\Vert F-F^{k,j-1}\Vert_F^2\},
    \end{equation}
     \end{small}
      \item[(f)] Update $\widehat{Y}^{k,j}$:
       \begin{small}
    \begin{equation}\label{update_Y_ba}
    \widehat{Y}^{k,j}\in \mathop{\arg\min}\;\{L(W^{k,j},U^{k,j},V^{k,j},Y^{k,j},F^{k,j},\widehat{Y},\overline{\lambda}^{k};\rho^k)+
    \frac{C_6^{k,j-1}}{2}\Vert \widehat{Y}-\widehat{Y}^{k,j-1}\Vert_F^2\},
    \end{equation}
     \end{small}
\end{itemize}
where the proximal parameters $\{C_i^{k,j}\}_{k,j}$ need to satisfy
\begin{equation*}
    0< \underline{C}\leq C_i^{k,j}< \overline{C}< \infty,\ k,j\in\mathbb{N},\ i=1,2,3,4,5,6,
\end{equation*}
for some predetermined positive constants $\underline{C}$ and $\overline{C}$.

By direct calculation, the subproblems in (\ref{update_W}-\ref{update_Y_ba}) have closed-form solutions as follows:
\begin{itemize}
    \item[(a)] For  (\ref{update_W}):
    $$W^{k,j}=\left(\frac{1}{a}I_d-\frac{\rho^k}{a^2}X(I_n+\frac{\rho^k}{a}X^{\intercal}X)^{-1}X^{\intercal}\right)Z,$$
        where   $a=2\gamma+\rho^k+C_1^{k,j-1}$ and
      $$Z=X\overline{ \lambda}_1^k+\overline{ \lambda}_2^k+\rho^{k}XY^{k,j-1}-\rho^{k}XU^{k,j-1}+\rho^{k}V^{k,j-1}+C_1^{k,j-1}W^{k,j-1}.$$
 \item[(b)] For  (\ref{update_U}):
    $U^{k,j}=(U_i^{k,j})_{i\in[n]}$, where $U_i^{k,j}$ is the row vector of  $U^{k,j}$.

    Set $$N=Y^{k,j-1}-X^\intercal W^{k,j}+\frac{\overline{\lambda}_1^k}{\rho^k}$$ and denote $n_i$ as the $i$-th row vector of $N$. Then,
    \begin{equation*}
    \begin{aligned}
    U_i^{k,j}=\max\Big\{&0, 1-\frac{\alpha}{\Vert\rho^kn_i+C_2^{k,j-1}U_i^{k,j-1}\Vert_2}\Big\}\left[\frac{\rho^k}{\rho^k+C_2^{k,j-1}}n_i+\frac{C_2^{k,j-1}}{\rho^k+C_2^{k,j-1}}U_i^{k,j-1}\right].
    \end{aligned}
    \end{equation*}
 \item[(c)] For  (\ref{update_V}):
      $V^{k,j}=(V_i^{k,j})_{i\in[d]}$ , where  $V_i^{k,j}$
       is the row vector of $V^{k,j}$.

      Set $$M=W^{k,j}-\frac{\overline{\lambda}_2^k}{\rho^k}$$ and its row vector is denoted by $m_i$. Then,
      \begin{equation*}
       \begin{aligned}
    	V_i^{k,j}=\max\Big\{&0, 1-\frac{\beta}{\Vert\rho^km_i+C_3^{k,j-1}V_i^{k,j-1}\Vert_2}\Big\}\left[\frac{\rho^k}{\rho^k+C_3^{k,j-1}}m_i+\frac{C_3^{k,j-1}}{\rho^k+C_3^{k,j-1}}V_i^{k,j-1}\right].
       \end{aligned}
       \end{equation*}
\item[(d)] For (\ref{update_Y}):\\
       $$Y^{k,j}=[2L+(3\rho^k+C_4^{k,j-1})I]^{-1}P,$$ where $$P=\overline{\lambda}_4^k-\overline{ \lambda}_3^k-\overline{\lambda}_1^k+\rho^kX^{\intercal}W^{k,j}+\rho^kU^{k,j}+\rho^kF^{k,j-1}+\rho^k\widehat{Y}^{k,j-1}+C_4^{k,j-1}Y^{k,j-1}.$$
\item[(e)] For  (\ref{update_F}):\\
      $F^{k,j}=(F_{s,t}^{k,j})_{s\in[n],t\in[c]}$ and  $F_{s,t}^{k,j}=\Pi_{[0,1]}A_{st},$ where  $$A=(A_{s,t})_{s\in[n],t\in[c]}=\frac{\rho^k(Y^{k,j}+\frac{\overline{\lambda}_3^k}{\rho^k})+C_5^{k,j-1}F^{k,j-1}}{\rho^k+C_5^{k,j-1}}.$$
\item[(f)] For  (\ref{update_Y_ba}):\\
     $\widehat{Y}^{k,j}=UI_{n\times c}V^{\intercal}$, where $U\in \mathbb{R}^{n\times n}, V\in \mathbb{R}^{c\times c}$ are two orthogonal matrices and $\sum\in\mathbb{R}^{n\times c}$ is a diagonal matrix satisfying the \text{SVD} factorization
      \begin{equation*}
 	  \frac{\rho^k(Y^{k,j}-\frac{\overline{\lambda}_4^k}{\rho^k})+C_6^{k,j-1}\widehat{Y}^{k,j-1}}{\rho^k+C_6^{k,j-1}}=U\sum V^{\intercal}.
      \end{equation*}
\end{itemize}

The iteration is terminated if there exists $\Theta^{k,j}\in\partial L(W^{k,j},U^{k,j},V^{k,j},Y^{k,j},F^{k,j},\widehat{Y}^{k,j},\overline{\lambda}^k;\rho^k)$ satisfying
\begin{equation*}
   \Vert\Theta^{k,j}\Vert_{\infty}\leq\epsilon_k,\quad  \textbf{O}\leq F^{k,j}\leq \textbf{E}, \quad (\widehat{Y}^{k,j})^{\intercal}\widehat{Y}^{k,j}=I_c,
\end{equation*}
where $\Theta^{k,j}:=(\Theta_1^{k,j},\Theta_2^{k,j},\Theta_3^{k,j},\Theta_4^{k,j},\Theta_5^{k,j},\Theta_6^{k,j})\in\mathbb{R}^{d\times c}\times\mathbb{R}^{n\times c}\times\mathbb{R}^{d\times c}\times\mathbb{R}^{n\times c}\times\mathbb{R}^{n\times c}\times \mathbb{R}^{n\times c}$ is concretely expressed in the form
 \begin{equation}\label{theta}
 \left\{
\begin{aligned}
    \Theta_1^{k,j}:=&\rho^k X(Y^{k,j-1}-Y^{k,j})+\rho^k X(U^{k,j}-U^{k,j-1})+\rho^k(V^{k,j-1}-V^{k,j})\\&+
    C_1^{k,j-1}(W^{k,j-1}-W^{k,j})\\
    \Theta_2^{k,j}:=&\rho^k(Y^{k,j-1}-Y^{k,j})+C_2^{k,j-1}(U^{k,j-1}-U^{k,j})\\
    \Theta_3^{k,j}:=&C_3^{k,j-1}(V^{k,j-1}-V^{k,j})\\
    \Theta_4^{k,j}:=&\rho^k(F^{k,j-1}-F^{k,j})+\rho^k(\widehat{Y}^{k,j-1}-\widehat{Y}^{k,j})+C_4^{k,j-1}(Y^{k,j-1}-Y^{k,j})\\
    \Theta_5^{k,j}:=&C_5^{k,j-1}(F^{k,j-1}-F^{k,j})\\
    \Theta_6^{k,j}:=&C_6^{k,j-1}(\widehat{Y}^{k,j-1}-\widehat{Y}^{k,j}).
\end{aligned}
\right.
\end{equation}

We summarize the algorithmic framework of PAM in Algorithm \ref{algo2}, whose convergence analysis is established in the next subsection.
\begin{algorithm}
\caption{PAM Method for (\ref{sublagpro}-\ref{Critical})}\label{algo2}
\begin{algorithmic}
 \STATE {{\bf Input:}\\
Let
 ($W^{1,0}$,$U^{1,0}$,$V^{1,0}$,$Y^{1,0}$,$F^{1,0}$,$\widehat{Y}^{1,0}$)
 be any initialization;\\ For $k\!\geq\!2$, set
 $(W^{k,0},U^{k,0},V^{k,0},Y^{k,0},F^{k,0},\widehat{Y}^{k,0})=(W^{k-1},U^{k-1},V^{k-1},Y^{k-1},F^{k-1},\widehat{Y}^{k-1})$;}
 \STATE {{\bf Output:}
 ($W^k$,$U^k$,$V^k$,$Y^k$,$F^k$,$\widehat{Y}^k$);}
 \STATE \textbf{Step 1:} Reiterate on $j$ until $\Vert\Theta^{k,j}\Vert_{\infty}\leq\epsilon_k$, where $\Theta^{k,j}$ is defined by (\ref{theta});\\
 \begin{itemize}
    \item[1.]Compute $W^{k,j}$ by (\ref{update_W});
    \item[2.] Compute $U^{k,j}$ by (\ref{update_U});
    \item[3.] Compute $V^{k,j}$ by (\ref{update_V});
    \item[4.] Compute $Y^{k,j}$ by (\ref{update_Y});
    \item[5.] Compute $F^{k,j}$ by (\ref{update_F});
    \item[6.] Compute $\widehat{Y}^{k,j}$ by (\ref{update_Y_ba});
 \end{itemize}
\STATE \textbf{Step 2:} Set
\begin{equation*}
 (W^k,U^k,V^k,Y^k,F^k,\widehat{Y}^k):=
 (W^{k,j},U^{k,j},V^{k,j},Y^{k,j},F^{k,j},\widehat{Y}^{k,j})
\end{equation*}
 and $\Theta^k:=\Theta^{k.j}$.
\end{algorithmic}	
\end{algorithm}

\subsection{Convergence Analysis for Algorithm \ref{algo2}}

For the sake of notation simplicity, we fix some notations. We define $T:=(W,U,V,Y,F,\widehat{Y})$ and $L_k(T):=L(W,U,V,Y,F,\widehat{Y},\overline{\lambda}^k;\rho^k)$ for the $k$-th outer iteration. In this part, we will establish the global convergence for Algorithm \ref{algo2}, in other words, we can derive that the solution set of (\ref{sublagpro}-\ref{Critical}) is nonempty and hence Algorithm \ref{algo1} is well defined with using Algorithm \ref{algo2} to solve the subproblem in Step 1.

We first claim that $\Theta^{k,j}:=(\Theta_1^{k,j},\Theta_2^{k,j},\Theta_3^{k,j},\Theta_4^{k,j},\Theta_5^{k,j},\Theta_6^{k,j})$ defined by (\ref{theta}) must satisfy
 \begin{equation*}
     \Theta^{k,j}\in\partial L_k(T^{k,j})\quad \forall\,j\in\mathbb{N}.
 \end{equation*}
 for each $k\in\mathbb{N}$.

 Considering the structure of  $L_k(T)$, it can be split as
\begin{equation}
\begin{aligned}\label{def_Lk}
  L_k(T)=f_1(W)+f_2(Y)+f_3(U)+f_4(V)
  +f_5(F)+f_6(\widehat{Y})+g_k(T),
  \end{aligned}
\end{equation}
where
\begin{equation*}
    \left\{
    \begin{aligned}
        f_1(W):=&\gamma\Vert W\Vert_F^2;\qquad f_2(Y):=\text{Tr}(Y^{\intercal}LY);\qquad f_3(U):=\alpha\Vert U\Vert_{2,1};\\
        f_4(V):=&\beta\Vert V\Vert_{2,1};\qquad f_5(F):=\delta_{\mathbb{S}_2}(F);\qquad f_6(\widehat{Y}):=\delta_{\mathbb{S}_1}(\widehat{Y});\\
        g_k(T):=&\langle\overline{\lambda}_1^k,Y-X^{\intercal}W-U\rangle+
                                  \langle\overline{\lambda}_2^k,V-W\rangle+
                                  \langle\overline{\lambda}_3^k,Y-F\rangle+\langle\overline{\lambda}_4^k,\widehat{Y}-Y\rangle\\&+\frac{\rho^k}{2}\Big(\Vert Y-X^{\intercal}W-U\Vert_F^2
                +\Vert \widehat{Y}-Y\Vert_F^2
                +\Vert V-W\Vert_F^2+ \Vert Y-F\Vert_F^2\Big).
    \end{aligned}
    \right.
\end{equation*}
Then, a direct calculation shows that $\Theta^{k,j}:=(\Theta_1^{k,j},\Theta_2^{k,j},\Theta_3^{k,j},\Theta_4^{k,j},\Theta_5^{k,j},\Theta_6^{k,j})$ defined by (\ref{theta}) can be expressed in terms of partial derivatives of $g:=g_k$ as
\begin{small}
\begin{equation}\label{theta2}
    \left\{
    \begin{aligned}
     \Theta_1^{k,j}=&-\nabla_W g(W^{k,j},U^{k,j-1},V^{k,j-1},Y^{k,j-1},F^{k,j-1},\widehat{Y}^{k,j-1})-C_1^{k,j-1}(W^{k,j}-W^{k,j-1})+\nabla_W g(T^{k,j}),\\
     \Theta_2^{k,j}=&-\nabla_U g(W^{k,j},U^{k,j},V^{k,j-1},Y^{k,j-1},F^{k,j-1},\widehat{Y}^{k,j-1})-C_2^{k,j-1}(U^{k,j}-U^{k,j-1})+\nabla_U g(T^{k,j}),\\
     \Theta_3^{k,j}=&-\nabla_V g(W^{k,j},U^{k,j},V^{k,j},Y^{k,j-1},F^{k,j-1},\widehat{Y}^{k,j-1})-C_3^{k,j-1}(V^{k,j}-V^{k,j-1})+\nabla_V g(T^{k,j}),\\
     \Theta_4^{k,j}=&-\nabla_Y g(W^{k,j},U^{k,j},V^{k,j},Y^{k,j},F^{k,j-1},\widehat{Y}^{k,j-1})-C_4^{k,j-1}(Y^{k,j}-Y^{k,j-1})+\nabla_Y g(T^{k,j}),\\
     \Theta_5^{k,j}=&-\nabla_F g(W^{k,j},U^{k,j},V^{k,j},Y^{k,j},F^{k,j},\widehat{Y}^{k,j-1})-C_5^{k,j-1}(F^{k,j}-F^{k,j-1})+\nabla_F g(T^{k,j}),\\
     \Theta_6^{k,j}=&-\nabla_{\widehat{Y}} g(W^{k,j},U^{k,j},V^{k,j},Y^{k,j},F^{k,j},\widehat{Y}^{k,j})-C_6^{k,j-1}(\widehat{Y}^{k,j}-\widehat{Y}^{k,j-1})+\nabla_{\widehat{Y}} g(T^{k,j}).
    \end{aligned}
    \right.
\end{equation}
\end{small}
Moreover, given \begin{footnotesize}
$(W^{k,j-1},U^{k,j-1},V^{k,j-1},Y^{k,j-1},F^{k,j-1},\widehat{Y}^{k,j-1})$\end{footnotesize}, using 8.8(c) in \citet{rockafellar2009variational}, the necessary first-order optimality conditions for the subproblem (\ref{update_W}-\ref{update_Y_ba}) are the following system:
\begin{small}
\begin{equation}\label{opti}
\left\{
\begin{aligned}
 &\nabla_W g(W^{k,j},U^{k,j-1},V^{k,j-1},Y^{k,j-1},F^{k,j-1},\widehat{Y}^{k,j-1})+\nabla f_1(W^{k,j})+C_1^{k,j-1}(W^{k,j}-W^{k,j-1})=\textbf{O},\\
 &\xi^{k,j}+\nabla_U g(W^{k,j},U^{k,j},V^{k,j-1},Y^{k,j-1},F^{k,j-1},\widehat{Y}^{k,j-1})+C_2^{k,j-1}(U^{k,j}-U^{k,j-1})=\textbf{O},\\
 &\zeta^{k,j}+\nabla_V g(W^{k,j},U^{k,j},V^{k,j},Y^{k,j-1},F^{k,j-1},\widehat{Y}^{k,j-1})+C_3^{k,j-1}(V^{k,j}-V^{k,j-1})=\textbf{O},\\
 &\nabla f_2(Y^{k,j})+\nabla_Y g(W^{k,j},U^{k,j},V^{k,j},Y^{k,j},F^{k,j-1},\widehat{Y}^{k,j-1})+C_4^{k,j-1}(Y^{k,j}-Y^{k,j-1})=\textbf{O},\\
 &\vartheta^{k,j}+\nabla_F g(W^{k,j},U^{k,j},V^{k,j},Y^{k,j},F^{k,j},\widehat{Y}^{k,j-1})+C_5^{k,j-1}(F^{k,j}-F^{k,j-1})=\textbf{O},\\
 &\varsigma^{k,j}+\nabla_{\widehat{Y}} g(W^{k,j},U^{k,j},V^{k,j},Y^{k,j},F^{k,j},\widehat{Y}^{k,j})+C_6^{k,j-1}(\widehat{Y}^{k,j}-\widehat{Y}^{k,j-1})=\textbf{O},
\end{aligned}
\right.
\end{equation}
\end{small}
where $\xi^{k,j}\in\partial f_3(U^{k,j})$, $\zeta^{k,j}\in\partial f_4(V^{k,j})$, $\vartheta^{k,j}\in\partial f_5(F^{k,j})$, and $\varsigma^{k,j}\in\partial f_6(\widehat{Y}^{k,j})$. Combining (\ref{theta2}) with (\ref{opti}), we have
\begin{equation*}
    \left\{
    \begin{aligned}
    &\Theta_1^{k,j}=\nabla f_1(W^{k,j})+\nabla_W g(T^{k,j}),\\
    &\Theta_2^{k,j}=\xi^{k,j}+\nabla_U g(T^{k,j}),\\
    &\Theta_3^{k,j}=\zeta^{k,j}+\nabla_V g(T^{k,j}),\\
    &\Theta_4^{k,j}=\nabla f_2(Y^{k,j})+\nabla_Y g(T^{k,j}),\\
    &\Theta_5^{k,j}=\vartheta^{k,j}+\nabla_F g(T^{k,j}),\\
    &\Theta_6^{k,j}=\varsigma^{k,j}+\nabla_{\widehat{Y}} g(T^{k,j}).\\
    \end{aligned}
    \right.
\end{equation*}
By Proposition 2.1 in \cite{attouch2010proximal}, for each $k\in\mathbb{N}$, we get
\begin{equation*}
    \Theta^{k,j}\in\partial L_k(W^{k,j},U^{k,j},V^{k,j},Y^{k,j},F^{k,j},\widehat{Y}^{k,j}),\quad \forall\,j\in\mathbb{N}.
\end{equation*}

Thus, we can obtain the following theorem which shows that Algorithm \ref{algo2} converges, which means the Step 1 of Algorithm \ref{algo1} is well defined. The proof is based on a general result established in \citet[Theorem 6.2]{attouch2013convergence}.
\begin{theorem}\label{theo}
Set parameters $r> 1, \rho^1> 0$ in Algorithm $\ref{algo1}$. For each $k\in\mathbb{N}$, we have the sequence $\{T^{k,j}\}_{j\in\mathbb{N}}$ produced by Algorithm \ref{algo2} converges and \begin{equation*}
    \Vert\Theta^{k,j}\Vert_{\infty}\rightarrow 0 \quad as\;j\rightarrow\infty.
\end{equation*}
\end{theorem}
\begin{proof}
We know that $\mathbb{S}_1$ and $\mathbb{S}_2$ are semi-algebraic sets and their indicator functions are semi-algebraic \citep{bolte2014proximal}. The quadratic functions $\textbf{x}^\intercal L\textbf{x}$ and $\Vert\textbf{x}\Vert_p (\text{p\;is\;rational})$ are also semi-algebraic. Using the fact that composition of semi-algebraic functions is semi-algebraic, we derive that $L_k$ is a semi-algebraic function. Also known is that the semi-algebraic function is a Kurdyka-{\L}ojasiewicz (KL) function~\cite[Appendic]{bolte2014proximal}. Thus, $L_k$ is a KL function. From the expression (\ref{def_Lk}) of $L_k$, it can be seen that the function $L_k$ satisfies: (i)$f_i$ is a proper lower semicontinuous function, $i=1,2,3,4,5,6$; (ii) $g_k$ is a $C^1$-function with locally Lipschitz continuous gradient.

Next, we will verify that for each $k\in\mathbb{N}$, $L_k$ is bounded below and the sequence $\{T^{k,j}\}_{j\in\mathbb{N}}$ is bounded.
For each $k\in\mathbb{N}$, the lower boundness of $L_k$ is proved by showing that $L_k$ is a coercive function (i.e., $L_k(T)\rightarrow +\infty$ when $\Vert T\Vert_{\infty}\rightarrow\infty$), provided that the parameters $r>1, \rho^1>0$. Clearly, the five terms $f_1, f_3, f_4, f_5, f_6$ of $L_k$ in (\ref{def_Lk}) are coercive. Then the residual terms are
\begin{equation*}\label{f2+gk}
\begin{aligned}
f_2(Y)+g_k(W,U,V,Y,F,\widehat{Y})
=&\text{Tr}(Y^\intercal LY)+\langle\overline{\lambda}_1^k,Y-X^{\intercal}W-U\rangle+\langle\overline{\lambda}_2^k,V-W\rangle
+                           \langle\overline{\lambda}_3^k,Y-F\rangle\\&+\langle\overline{\lambda}_4^k,\widehat{Y}-Y\rangle+\frac{\rho^k}{2}\Big(\Vert \widehat{Y}-Y\Vert_F^2+
                                   \Vert V-W\Vert_F^2+\Vert Y-F\Vert_F^2\\&+\Vert Y-X^{\intercal}W-U\Vert_F^2\Big).
\end{aligned}
\end{equation*}
We can rewrite it as
\begin{equation*}
\begin{aligned}
    f_2(Y)+g_k(W,U,V,Y,F,\widehat{Y})
    =g_{1,k}(W,U,Y)+g_{2,k}(W,V,Y,F,\widehat{Y}),
    \end{aligned}
\end{equation*}
where
\begin{equation*}
\begin{aligned}
   g_{1,k}(W,U,Y)=&\text{Tr}(Y^\intercal LY)+\langle\overline{\lambda}_1^k,Y-X^{\intercal}W-U\rangle+\frac{\rho^k}{2}\Vert Y-X^{\intercal}W-U\Vert_F^2
   \end{aligned}
\end{equation*}
and
\begin{equation*}
\begin{aligned}
    g_{2,k}(W,V,Y,F,\widehat{Y})=&\langle\overline{\lambda}_2^k,V-W\rangle+\langle\overline{\lambda}_3^k,Y-F\rangle+\langle\overline{\lambda}_4^k,\widehat{Y}-Y\rangle+\frac{\rho^k}{2}(\Vert \widehat{Y}-Y\Vert_F^2+
           \Vert V-W\Vert_F^2\\&+\Vert Y-F\Vert_F^2).
    \end{aligned}
\end{equation*}
Let us observe that
\begin{equation*}
\begin{aligned}
g_{1,k}(W,U,Y)=&\text{Tr}(Y^\intercal LY)+\frac{\rho^k}{2}\Vert Y-X^{\intercal}W-U+\frac{\overline{\lambda}_1^k}{\rho^k}\Vert_F^2-\frac{\rho^k}{2}\Vert\frac{\overline{\lambda}_1^k}{\rho^k}\Vert_F^2
\end{aligned}
\end{equation*}
and
\begin{equation*}
\begin{aligned}
g_{2,k}(W,V,Y,F,\widehat{Y})=&\frac{\rho^k}{2}\Big[\Vert V-W+\frac{\overline{\lambda}_2^k}{\rho^k}\Vert_F^2+\Vert Y-F+\frac{\overline{\lambda}_3^k}{\rho^k}\Vert_F^2+\Vert\widehat{Y}-Y+\frac{\overline{\lambda}_4^k}{\rho^k}\Vert_F^2-(\Vert\frac{\overline{\lambda}_2^k}{\rho^k}\Vert_F^2\\&+\Vert\frac{\overline{\lambda}_3^k}{\rho^k}\Vert_F^2+\Vert\frac{\overline{\lambda}_4^k}{\rho^k}\Vert_F^2)\Big].
\end{aligned}
\end{equation*}
Thus, $g_{1,k}(W,U,Y)$ and $g_{2,k}(W,V,Y,F,\widehat{Y})$ are all bounded below. Furthermore, the functions $\{L_k\}_k\in\mathbb{N}$ defined by (\ref{def_Lk}) are all coercive.

The boundedness of the sequence $\{T^{k,j}\}_{j\in\mathbb{N}}$ is proved by contradiction. On the one hand, suppose that the sequence $\{T^{k_0,j}\}_{j\in\mathbb{N}}$ is unbounded, and so $\lim_{j\rightarrow\infty}\Vert T^{k_0,j}\Vert=\infty$. Then, it follows from the coercive of $L_{k_0}(T)$ that the sequence $\{L_{k_0}(T^{k_0,j})\}_{j\in\mathbb{N}}$ should diverge to infinity. On the other hand, let
\begin{align*}
&\widetilde{L}^1_{k_0,j}\!=\!L_{k_0}(W^{k_0,j+1}, U^{k_0,j}, V^{k_0,j}, Y^{k_0,j}, F^{k_0,j}, \widehat{Y}^{k_0,j}),\\ &\widetilde{L}^2_{k_0,j}\!=\!L_{k_0}(W^{k_0,j+1}, U^{k_0,j+1}, V^{k_0,j}, Y^{k_0,j}, F^{k_0,j}, \widehat{Y}^{k_0,j}),\\ &\widetilde{L}^3_{k_0,j}\!=\!L_{k_0}(W^{k_0,j+1}, U^{k_0,j+1}, V^{k_0,j+1}, Y^{k_0,j}, F^{k_0,j}, \widehat{Y}^{k_0,j}),\\ &\widetilde{L}^4_{k_0,j}\!=\!L_{k_0}(W^{k_0,j+1}, U^{k_0,j+1}, V^{k_0,j+1}, Y^{k_0,j+1}, F^{k_0,j}, \widehat{Y}^{k_0,j}),\\ &\widetilde{L}^5_{k_0,j}\!=\!L_{k_0}(W^{k_0,j+1}, U^{k_0,j+1}, V^{k_0,j+1}, Y^{k_0,j+1}, F^{k_0,j+1}, \widehat{Y}^{k_0,j}).
\end{align*}
 By (\ref{update_W}-\ref{update_Y_ba}), we deduce that
\begin{align*}
    &\widetilde{L}^1_{k_0,j}+\frac{C_1^{k_0,j}}{2}\Vert W^{k_0,j+1}-W^{k_0,j}\Vert_F^2
    \leq L_{k_0}(T^{k_0,j});\\
    &\widetilde{L}^2_{k_0,j}+\frac{C_2^{k_0,j}}{2}\Vert U^{k_0,j+1}-U^{k_0,j}\Vert_F^2
    \leq \widetilde{L}^1_{k_0,j};\\
    &\widetilde{L}^3_{k_0,j}+\frac{C_3^{k_0,j}}{2}\Vert V^{k_0,j+1}-V^{k_0,j}\Vert_F^2
    \leq \widetilde{L}^2_{k_0,j};\\
    &\widetilde{L}^4_{k_0,j}+\frac{C_4^{k_0,j}}{2}\Vert Y^{k_0,j+1}-Y^{k_0,j}\Vert_F^2
    \leq \widetilde{L}^3_{k_0,j};\\
    &\widetilde{L}^5_{k_0,j}+\frac{C_5^{k_0,j}}{2}\Vert F^{k_0,j+1}-F^{k_0,j}\Vert_F^2
    \leq \widetilde{L}^4_{k_0,j};\\
    &L_{k_0}(T^{k_0,j+1})+\frac{C_6^{k_0,j}}{2}\Vert \widehat{Y}^{k_0,j+1}-\widehat{Y}^{k_0,j}\Vert_F^2
    \leq \widetilde{L}^5_{k_0,j}.
\end{align*}
Summing up these inequalities, we have
\begin{equation*}
    L_{k_0}(T^{k_0,j+1})+\frac{\underline{C}}{2}\Vert T^{k_0,j+1}-T^{k_0,j}\Vert_F^2\leq L_{k_0}(T^{k_0,j}),\ j\in\mathbb{N},
\end{equation*}
which implies that $\{L_{k_0}(T^{k_0,j})\}_{j\in\mathbb{N}}$ is a nonincreasing sequence, leading to a contradiction.

Based on a general result established in \citet[Theorem 6.2]{attouch2013convergence}, we know that for each $k\in\mathbb{N}$, the sequence $\{T^{k,j}\}_{j\in\mathbb{N}}$ has finite length, i.e., $\sum_{j=1}^{\infty}\Vert T^{k,j+1}-T^{k,j}\Vert_F<\infty$, and the sequence $\{T^{k,j}\}_{j\in\mathbb{N}}$ converges to a critical point of $L_k$. Since $\Theta^{k,j}$ is given by (\ref{theta}), we conclude that for each $k\in\mathbb{N}$, $\Vert\Theta^{k,j}\Vert_{\infty}\rightarrow 0$ as $j\rightarrow\infty$.
The proof is complete.
\end{proof}
\section{Convergence Analysis of Our Inexact ALM Method}

In this section, we discuss the convergence for our inexact ALM method given in Algorithm \ref{algo1}.

In the following, we rewrite (\ref{pro2}) using the notation of vectors. Let $\textbf{x}\in\mathbb{R}^{2dc+4nc}$ denote the column vector formed by concatenating the columns of
$W, U, V, Y, F, \widehat{Y}$, i.e.,
\begin{equation}\label{defx}
    \textbf{x}:=\text{Vec}([W\vert U\vert V\vert Y\vert F\vert \widehat{Y}]).
\end{equation}
Then, problem (\ref{pro2}) can be rewritten as follows:
\begin{equation}\label{vecpro}
\begin{aligned}
\min_{\textbf{x}\in\mathbb{R}^{2dc+4nc}}&f(x)\quad\text{s.t.}\ h_1(\textbf{x})=0\;\text{and}\;h_2(\textbf{x})=0,
\end{aligned}
\end{equation}
where $h_1(\textbf{x})\in\mathbb{R}^{3nc+dc}$ denotes $\text{Vec}([Y-X^\intercal W-U\vert V-W\vert Y-F\vert \widehat{Y}-Y])$, $h_2(\textbf{x})$ denotes the $\frac{c(c+1)}{2}\times 1$ vector obtained by vectorizing only the lower triangular part of the symmetric matrix $\widehat{Y}^\intercal \widehat{Y}-I_c$, and

\begin{equation*}
\begin{aligned}
f(\textbf{x}):=&\sum_{j=1}^{c} Y_j^\intercal LY_j+\gamma\Vert W_j\Vert_2^2+\delta_{S'}(F_j)+\sum_{i=1}^{n}\alpha\Vert U_i\Vert_2+\sum_{i=1}^d\beta\Vert V_i\Vert_2.
\end{aligned}
\end{equation*}

In this case, $Y_j, W_j, F_j$ are the column vectors of $Y, W$ and $F$, respectively; $U_i$ and $V_i$ are the row vectors of $U$ and $V$, respectively; $S'=\{F_j\;\vert\; \textbf{0}\leq F_j\leq \textbf{1}\}$. Let $\Lambda:=\text{Vec}([\lambda_1\vert \lambda_2\vert \lambda_3\vert \lambda_4])$. Then, the corresponding augmented Lagrangian function of (\ref{vecpro}) is
\begin{equation*}
    L(\textbf{x},\Lambda;\rho):= f(\textbf{x})+\sum_{i=1}^{m_1}[\Lambda]_i[h_1(\textbf{x})]_i+\frac{\rho}{2}\sum_{i=1}^{m_1}[h_1(\textbf{x})]_i^2,
\end{equation*}
where $\textbf{x}\in\Gamma$, $m_1:=3nc+dc$, $m_2:=\frac{c(c+1)}{2}$ and
\begin{equation}
    \Gamma:=\{\textbf{x}\;\vert\;h_2(\textbf{x})=0\}.
\end{equation}
Therefore, $(W^*,U^*,V^*,Y^*,F^*,\widehat{Y}^*)$ is a $\text{KKT}$ point for optimization problem (\ref{pro2}) if and only if the vector $\textbf{x}$ defined by (\ref{defx}) is a $\text{KKT}$ point for optimization problem (\ref{vecpro}), i.e., there exist $\theta^*\in\partial f(\textbf{x}^*)$, $\Lambda^*\in\mathbb{R}^{m_1}$, $\eta^*\in\mathbb{R}^{m_2}$ such that the following system is fulfilled
\begin{equation}
\left\{
\begin{array}{l}
\theta^*+\sum_{i=1}^{m_1}[\Lambda^*]_i\nabla[h_1(\textbf{x}^*)]_i+\sum_{i=1}^{m_2}[\eta^*]_i\nabla[h_2(\textbf{x}^*)]_i=0,\\
h_1(\textbf{x}^*)=0,\\
h_2(\textbf{x}^*)=0.\\
\end{array}
\right.
\end{equation}

Suppose that $\{T^k\}_{k\in\mathbb{N}}$ is a sequence generated by Algorithm \ref{algo1}. We will show first that the sequence $\{T^k\}_{k\in\mathbb{N}}$ is bounded. Then, there exists at least one convergent subsequence of $\{T^k\}_{k\in\mathbb{N}}$. We will next show that it converges to a $\text{KKT}$ point of the optimization problem (\ref{vecpro}). Thus, we have the following main convergence result for Algorithm \ref{algo1}.
\begin{theorem}\label{congver}
Suppose that the parameters $r> 1$ and $\rho^1> 0$ in Algorithm \ref{algo1}. Let $\{T^k\}_{k\in\mathbb{N}}$ be the sequence generated by Algorithm \ref{algo1}. Then, the limit point set of $\{T^k\}_{k\in\mathbb{N}}$ is nonempty, and every limit point is a $\text{KKT}$ point of the original problem (\ref{pro2}).
\end{theorem}

 To show Theorem \ref{congver}, we need the following two lemmas.
\begin{lemma}\label{sequbound}
Let $\{T^k\}_{k\in\mathbb{N}}$ be the sequence generated by Algorithm \ref{algo1}. Suppose that the parameters $r, \rho^1$ in Algorithm \ref{algo1} are chosen so that $r> 1$ and $\rho^1>0$. Then, $\{T^k\}_{k\in\mathbb{N}}$ is bounded and thus contains at least one convergent sequence.
\end{lemma}
\begin{proof}
It follows from (\ref{constrain}) that the sequence $\{F^k\}_{k\in\mathbb{N}}$ and $\{\widehat{Y}^k\}_{k\in\mathbb{N}}$ are bounded. The first four partial subdifferentials of $L$ in (\ref{Critical}) guarantee the following: there exist $\xi^k\in\partial\alpha\Vert U\Vert_{2,1}$, $\zeta^k\in\partial\beta\Vert V\Vert_{2,1}$ and $\aleph^k=(\aleph_1^k,\aleph_2^k,\aleph_3^k,\aleph_4^k)\in\mathbb{R}^{d\times c}\times\mathbb{R}^{n\times c}\times\mathbb{R}^{d\times c}\times\mathbb{R}^{n\times c}$ such that
\begin{equation}\label{sub}
    \left\{
    \begin{aligned}
    \aleph_1^k=&2\gamma W^k-X\overline{\lambda}_1^k-\overline{\lambda}_2^k-\rho^k X(Y^k-X^\intercal W^k-U^k)-\rho^k(V^k-W^k),\\
    \aleph_2^k=&\xi^k-\overline{\lambda}_1^k-\rho^k(Y^k-X^\intercal W^k-U^k),\\
    \aleph_3^k=&\zeta^k+\overline{\lambda}_2^k+\rho^k(V^k-W^k),\\
    \aleph_4^k=&2LY^k+\overline{\lambda}_1^k+\overline{\lambda}_3^k-\overline{\lambda}_4^k+\rho^k(Y^k-X^\intercal W^k-U^k)
               +\rho^k(Y^k-F^k)-\rho^k(\widehat{Y}^k-Y^k),
    \end{aligned}
    \right.
\end{equation}
where $\Vert\aleph^k\Vert_{\infty}\leq\epsilon^k$. By adding $\aleph_2^k$ and $\aleph_4^k$, we obtain that
\begin{equation*}
 \begin{aligned}
\aleph_2^k+\aleph_4^k=\xi^k+(2L+2\rho^k I)Y^k+\overline{\lambda}_3^k-\overline{\lambda}_4^k-\rho^k F^k-\rho^k\widehat{Y}^k.
\end{aligned}
\end{equation*}
This implies
\begin{equation}\label{Yk}
 Y^k=[2(L+\rho^k I)]^{-1}(\aleph_2^k+\aleph_4^k-\xi^k-\overline{\lambda}_3^k+\overline{\lambda}_4^k+\rho^k F^k+\rho^k\widehat{Y}^k).
\end{equation}
Let $L=D\text{diag}(\sigma_1,\cdots,\sigma_n)D^\intercal$ denotes the $\text{SVD}$ decomposition of the symmetric and positive semi-definite matrix $L$. Hence (\ref{Yk}) yields
\begin{equation}\label{finYk}
    \begin{aligned}
    Y^k=
    &D\text{diag}\left(\frac{1}{2(\sigma_1+\rho^k)},\frac{1}{2(\sigma_2+\rho^k)},\cdots,\frac{1}{2(\sigma_n+\rho^k)}\right)D^\intercal(\aleph_2^k+\aleph_4^k-\xi^k-\overline{\lambda}_3^k+\overline{\lambda}_4^k)\\&
    +D\text{diag}\left(\frac{\rho^k}{2(\sigma_1+\rho^k)},\frac{\rho^k}{2(\sigma_2+\rho^k)},\cdots,\frac{\rho^k}{2(\sigma_n+\rho^k)}\right)D^\intercal (F^k+\widehat{Y}^k).
    \end{aligned}
\end{equation}
Using the fact $\{\rho^k\}_{k\in\mathbb{N}}$ is a nondecreasing sequence and $2(L+\rho^1 I)\succ0$, for $k\in\mathbb{N}$, we have $2(L+\rho^k I)\succ0$, which derives $2(\sigma_i+\rho^k)> 0$, $i=1,2,\cdots,n$. Then, we can show that for each $k\in\mathbb{N}$
\begin{equation}\label{range}
\left\{
\begin{array}{l}
      0<\frac{1}{2(\sigma_i+\rho^k)}\leq\frac{1}{2(\sigma_i+\rho^1)}<+\infty,\quad i=1,2,\cdots,n;\\
      0<\frac{\rho^k}{2(\sigma_i+\rho^k)}\leq\frac{1}{2},\quad i=1,2,\cdots,n.
\end{array}
\right.
 \end{equation}
Note that $\{\xi^k\}_{k\in\mathbb{N}}$, $\{\aleph_2^k\}_{k\in\mathbb{N}}$, $\{\aleph_4^k\}_{k\in\mathbb{N}}$, $\{\overline{\lambda}_3^k\}_{k\in\mathbb{N}}$ and $\{\overline{\lambda}_4^k\}_{k\in\mathbb{N}}$ are bounded. It follows from (\ref{finYk}) and  (\ref{range}) that the sequence $\{Y^k\}_{k\in\mathbb{N}}$ is bounded.

Likewise, according to the expression of $\aleph_3^k$ and $\aleph_2^k$ in (\ref{sub}), we conclude that $\{\rho^k(V^k-W^k)\}_{k\in\mathbb{N}}$ and $\{\rho^k(Y^k-X^\intercal W^k-U^k)\}_{k\in\mathbb{N}}$ are bounded. Then, from the expression of $\aleph_1^k$ in (\ref{sub}), we must have that the sequence $\{W^k\}_{k\in\mathbb{N}}$ is bounded. Using the fact that $\rho^k\geq\rho^1$ again, we obtain that $\{V^k-W^k\}_{k\in\mathbb{N}}$ and $\{Y^k-X^\intercal W^k-U^k\}_{k\in\mathbb{N}}$ are bounded. Therefore, the sequence $\{V^k\}_{k\in\mathbb{N}}$ and $\{U^k\}_{k\in\mathbb{N}}$ are bounded. In a conclusion, the sequence $\{(W^k,U^k,V^k,Y^k,F^k,\widehat{Y}^k)\}_{k\in\mathbb{N}}$ is bounded. The proof is complete.
\end{proof}

\begin{lemma}\label{indepen}
Suppose that $\bar{\textbf{x}}\in\Gamma$. Then $\{\nabla[h_1(\bar{\textbf{x}})]_i\}_{i=1}^{m_1}\cup\{\nabla[h_2(\bar{\textbf{x}})]_i\}_{i=1}^{m_2}$ are linearly independent, where $h_1$ and $h_2$ are defined as in (\ref{vecpro}).
\end{lemma}
\begin{proof}
For convenience, we define the block diagonal matrix $A\in\mathbb{R}^{dc\times nc}$, $B\in\mathbb{R}^{dc\times dc}$ and $C\in\mathbb{R}^{nc\times nc}$ as follows:
\begin{equation*}
  A=
    \begin{bmatrix}
    -X&  &      & \\
      &-X&      & \\
      &  &\ddots&  \\
      &  &      &-X \\
    \end{bmatrix},\quad
 B=
 \begin{bmatrix}
 I_d&  &      & \\
  &I_d&      & \\
   &  &\ddots&  \\
    &  &      &I_d\\
 \end{bmatrix},\quad
 C=
 \begin{bmatrix}
 I_n&          &      & \\
 &I_n&      & \\
  &          &\ddots&  \\
   &          &      &I_n\\
 \end{bmatrix}.
\end{equation*}
By the structure of $\textbf{x}$ defined in (\ref{defx}), we have
 \begin{equation*}
    \nabla h_1(\textbf{x})=
    \left[
    \begin{array}{c|c|c|c}
    A&-B&\textbf{O}_{dc\times nc}&\textbf{O}_{dc\times nc}\\\hline
    -C&\textbf{O}_{nc\times dc}&\textbf{O}_{nc\times nc}&\textbf{O}_{nc\times nc}\\\hline
    \textbf{O}_{dc\times nc}&B&\textbf{O}_{dc\times nc}&\textbf{O}_{dc\times nc}\\\hline
    C&\textbf{O}_{nc\times dc}&C&-C\\\hline
    \textbf{O}_{nc\times nc}&\textbf{O}_{nc\times dc}&-C&\textbf{O}_{nc\times nc}\\\hline
    \textbf{O}_{nc\times nc}&\textbf{O}_{nc\times dc}&\textbf{O}_{nc\times nc}&C\\
    \end{array}
    \right]
\quad \text{and} \quad
\nabla h_2(\textbf{x})=
\left[
    \begin{array}{c}
  \textbf{O}_{dc\times m_2}\\\hline
  \textbf{O}_{nc\times m_2}\\\hline
  \textbf{O}_{dc\times m_2}\\\hline
  \textbf{O}_{nc\times m_2}\\\hline
  \textbf{O}_{nc\times m_2}\\\hline
  G(\textbf{x})
    \end{array}
    \right],
\end{equation*}
where $G(\textbf{x})$ is given in (\ref{longmatrix})
\begin{figure*}
\centering
\begin{scriptsize}
\begin{equation}\label{longmatrix}
 G(\textbf{x})=
    \left[
   \begin{array}{ccccccccccccc}
    2\widehat{Y}_1&\widehat{Y}_2&\widehat{Y}_3&\cdots&\widehat{Y}_c&\textbf{O}_{n\times 1}&\textbf{O}_{n\times 1}&\cdots&\textbf{O}_{n\times 1}&  &\textbf{O}_{n\times 1}&\textbf{O}_{n\times 1}&\textbf{O}_{n\times 1}\\
    \textbf{O}_{n\times 1}&\widehat{Y}_1&\textbf{O}_{n\times 1}&\cdots&\textbf{O}_{n\times 1}&2\widehat{Y}_2&\widehat{Y}_3&\cdots&\widehat{Y}_c&  &\vdots&\vdots&\vdots\\
    \textbf{O}_{n\times 1}&\textbf{O}_{n\times 1}&\widehat{Y}_1&\cdots&\textbf{O}_{n\times 1}&\textbf{O}_{n\times 1}&\widehat{Y}_2&\cdots&\textbf{O}_{n\times 1}&\cdots&\textbf{O}_{n\times 1}&\textbf{O}_{n\times 1}&\vdots\\
    \vdots&\vdots&\ddots&\ddots&\textbf{O}_{n\times 1}&\vdots&\ddots&\ddots&\vdots& &2\widehat{Y}_{c-1}&\widehat{Y}_c&\textbf{O}_{n\times 1}\\
    \textbf{O}_{n\times 1}&\textbf{O}_{n\times 1}&\cdots&\textbf{O}_{n\times 1}&\widehat{Y}_1&\textbf{O}_{n\times 1}&\cdots&\textbf{O}_{n\times 1}&\widehat{Y}_2& &\textbf{O}_{n\times 1}&\widehat{Y}_{c-1}&2\widehat{Y}_c
    \end{array}
    \right]
\end{equation}
\end{scriptsize}
\end{figure*}
and $\{\widehat{Y}_i\}_{i=1}^{c}$ are the cloumn vectors of $\widehat{Y}$.

As $\textbf{x} \in \Gamma$, we must have that the column vectors $\{\widehat{Y}_i\}_{i=1}^{c}$ are orthogonal to each other,
and then the columns of $G(\textbf{x})$ are orthogonal to each other.
Note that the first $3nc+2dc$ rows of $\nabla h_2(\textbf{x})$ constitute a zero matrix.
Therefore, it follows from the structure of $\nabla h_1(\textbf{x})$ and $\nabla h_2(\textbf{x})$ that $\{\nabla[h_1(\bar{\textbf{x}})]_i\}_{i=1}^{m_1}\cup\{\nabla[h_2(\bar{\textbf{x}})]_i\}_{i=1}^{m_2}$ are linearly independent for any $\textbf{x}\in\Gamma$. The proof is complete.
\end{proof}
By Lemmas \ref{sequbound} and \ref{indepen}, we can show that any accumulation point $\textbf{x}^*$ of the corresponding sequence $\{\textbf{x}^k\}_{k\in\mathbb{N}}$ with respect to $\{T^k\}_{k\in\mathbb{N}}$ is a $\text{KKT}$ point of problem (\ref{vecpro}). As shown in Remark \ref{rma2}, the normal cone $\partial\delta_{\mathbb{S}_1}(T)=N_{\mathbb{S}_1}(T)$ in vector notation is
\begin{equation*}
    N_{\Gamma}(\bar{\textbf{x}})=\{\nabla h_2(\bar{\textbf{x}})\upsilon| \upsilon\in\mathbb{R}^{m_2}\}=\{\sum_{i=1}^{m_2}[\upsilon]_i\nabla[h_2(\bar{\textbf{x}})]_i|\upsilon\in\mathbb{R}^{m_2}\}.
\end{equation*}
According to the well-definedness of (\ref{Critical}), in view of the vector notation, we can obtain a solution $\textbf{x}^k$ such that there exist two vectors $\theta^k\in\partial f(\textbf{x}^k)$ and  $\upsilon^k$ to satisfy
 \begin{equation*}
     \Vert \theta^k+\sum_{i=1}^{m_1}([\overline{\Lambda}^k]_i+\rho^k[h_1(\textbf{x}^k)]_i)\nabla[h_1(\textbf{x}^k)]_i+\sum_{i=1}^{m_2}[\upsilon^k]_i\nabla[h_2(\textbf{x}^k)]_i\Vert_\infty\leq\epsilon^k
 \end{equation*}
 for each $k\in\mathbb{N}$. The following result is central to this paper.
 \begin{theorem}\label{centraltheo}
 Let $\{\textbf{x}^k\}_{k\in\mathbb{N}}$ be the iteration sequence generated by Algorithm \ref{algo1} and $\textbf{x}^*$ be its accumulation point, i.e., there exists a subsequence $\mathcal{K}\subseteq \mathbb{N}$ such that $\lim_{k\in\mathcal{K}}\textbf{x}^k=\textbf{x}^*$. Then $\textbf{x}^*$ is also a $\text{KKT}$ point of problem (\ref{vecpro}).
 \end{theorem}
 \begin{proof}
 We first show that $\textbf{x}^*$ satisfies the feasibility of problem (\ref{vecpro}), i.e., $h_1(\textbf{x}^*)=0$ and $h_2(\textbf{x}^*)=0$. By (\ref{constrain}), we conclude that $h_2(\textbf{x}^k)=0$ for each $k\in\mathbb{N}$. The continuity of $h_2$ yields  $h_2(\textbf{x}^*)=0$, i.e., $\textbf{x}^*\in\Gamma$. The proof of  feasibility $h_1(\textbf{x}^*)=0$ is divided into two parts, according to the boundedness of the sequence $\{\rho^k\}_{k\in\mathbb{N}}$.

\textbf{Part I.} Suppose first that the penalty sequence $\{\rho^k\}_{k\in\mathbb{N}}$ is bounded. By the penalty parameter update rule (\ref{rho}), it follows that $\rho^k$ stabilizes after some $k_0$, which implies that $\Vert h_1(\textbf{x}^{k+1})\Vert_\infty\leq\tau\Vert h_1(\textbf{x}^{k})\Vert_\infty$ for all $k\geq k_0$ and the constant $\tau\in[0,1)$. By a standard continuity argument, we obtain that $h_1(\textbf{x}^*)=0$.

\textbf{Part II.} In the following, we assume that $\{\rho^k\}_{k\in\mathbb{N}}$ is unbounded. For each $k\in\mathcal{K}$, there exist vectors $\{\delta^k\}_{k\in\mathbb{N}}$ with $\Vert\delta^k\Vert_\infty\leq\epsilon^k$ and $\epsilon^k\downarrow 0$ such that
\begin{equation}\label{delta}
    \theta^k+\sum_{i=1}^{m_1}([\overline{\Lambda}^k]_i+\rho^k[h_1(\textbf{x}^k)]_i)\nabla[h_1(\textbf{x}^k)]_i+\sum_{i=1}^{m_2}[\upsilon^k]_i\nabla[h_2(\textbf{x}^k)]_i=\delta^k
\end{equation}
for some $\theta^k\in\partial f(\textbf{x}^k)$. Dividing both sides of (\ref{delta}) by $\rho^k$, we obtain that
\begin{equation}\label{divrho}
    \sum_{i=1}^{m_1}([\frac{\overline{\Lambda}^k}{\rho^k}]_i+[h_1(\textbf{x}^k)]_i)\nabla[h_1(\textbf{x}^k)]_i+\sum_{i=1}^{m_2}[\hat{\upsilon}^k]_i\nabla[h_2(\textbf{x}^k)]_i=\frac{\delta^k-\theta^k}{\rho^k},
\end{equation}
where $\hat{\upsilon}^k=\frac{\upsilon^k}{\rho^k}$. Define
\begin{equation*}
 H(\textbf{x})^\intercal:=[\nabla h_1(\textbf{x})\ \nabla h_2(\textbf{x})]
\end{equation*}
and
\begin{equation*}
\begin{aligned}
   \eta^k:=([\frac{\overline{\Lambda}^k}{\rho^k}]_1+[h_1(\textbf{x}^k)]_1,\cdots,[\frac{\overline{\Lambda}^k}{\rho^k}]_{m_1}+[h_1(\textbf{x}^k)]_{m_1},[\hat{\upsilon}^k]_1,
   \cdots,[\hat{\upsilon}^k]_{m_2})^\intercal.
   \end{aligned}
\end{equation*}
 Hence we can rewrite (\ref{divrho}) in the following way:
 \begin{equation*}
   H(\textbf{x})^\intercal \eta^k=\frac{\delta^k-\theta^k}{\rho^k}.
 \end{equation*}
A straightforward application of Lemma \ref{indepen} yields that $\{\nabla[h_1(\textbf{x}^*)]_i\}_{i=1}^{m_1}\cup\{\nabla[h_2(\textbf{x}^*)]_i\}_{i=1}^{m_2}$
are independent as $\textbf{x}^*\in\Gamma$. In addition, we notice that the gradient vectors $\nabla h_1, \nabla h_2$ are continuous and $h_2(\textbf{x}^k)=0$ for all $k\in\mathcal{K}$. This means that $H(\textbf{x}^k)\rightarrow H(\textbf{x}^*)$ and $H(\textbf{x}^*)$ has full rank as $\textbf{x}^*\in\Gamma$. Therefore, we have that $H(\textbf{x}^k)H(\textbf{x}^k)^\intercal\rightarrow H(\textbf{x}^*) H(\textbf{x}^*)^\intercal \succ 0$.
By the fact that eigenvalues of a symmetric matrix vary continuously with its matrix values. We then conclude that $H(\textbf{x}^k)H(\textbf{x}^k)^\intercal$ is nonsingular for sufficiently large $k\in\mathcal{K}$, which yields
\begin{equation*}
    \eta^k=[H(\textbf{x}^k)H(\textbf{x}^k)^\intercal]^{-1}H(\textbf{x}^k)\frac{\delta^k-\theta^k}{\rho^k}.
\end{equation*}
Since $f$ is a convex function, the set $\cup_{x\in\mathcal{X}}\partial f(\textbf{x})$ is bounded whenever $\mathcal{X}$ is bounded. A nice proof of this result can be found in \citet[Proposition B.24(b)]{bertsekas1999nonlinear}. It is then straightforward to see that $\{\theta^k\}_{k\in\mathcal{K}}$ is bounded by setting $\mathcal{X}=\{\textbf{x}^k\}_{k\in\mathcal{K}}$, where the boundedness of $\{\textbf{x}^k\}_{k\in\mathcal{K}}$ is motivated by Proposition \ref{sequbound}. Combining the previous result $\Vert\delta^k\Vert_{\infty}\leq\epsilon^k\downarrow 0$, we obtain for $k\in\mathcal{K}$
\begin{equation*}
    \eta^k\rightarrow 0\quad as\;k\rightarrow\infty.
\end{equation*}
Finally, with the boundedness of Lagrange multipliers $\{\overline{\lambda}^k\}_k$, $[h_1(\textbf{x}^*)]_i=0=[\hat{\upsilon}]_j$ is guaranteed for all $i,j$. Hence we conclude that  $h_1(\textbf{x}^*)=0$.

Next, we show that $\textbf{x}^*$ is a $\text{KKT}$ point. The boundedness of $\{\theta^k\}_{k\in\mathcal{K}}$ implies that there exists a subsequence $\mathcal{K}_1\subseteq\mathcal{K}$ such that $\lim_{k\in\mathcal{K}_1}\theta^k=\theta^*$. Together with $\lim_{k\in\mathcal{K}_1}\textbf{x}^k=\textbf{x}^*$ and $\theta^k\in\partial f(\textbf{x}^k)$, it can be follows from the closedness property of subdifferential that
\begin{equation*}
    \theta^*\in\partial f(\textbf{x}^*).
\end{equation*}
By the fact that $[\lambda^{k+1}]_i=[\overline{\Lambda}^k]_i+\rho^k[h_1(\textbf{x}^k)]_i$ for all $i$, we have that for $k\in\mathcal{K}_1$
\begin{equation}\label{finasub}
    \theta^k+\sum_{i=1}^{m_1}[\lambda^{k+1}]_i\nabla[h_1(\textbf{x}^k)]_i+\sum_{i=1}^{m_2}[\upsilon^k]_i\nabla[h_2(\textbf{x}^k)]_i=\delta^k
\end{equation}
for some vector $\delta^k$ with $\Vert\delta^k\Vert_\infty\leq\epsilon^k\downarrow 0$ and $\theta^k\in\partial f(\textbf{x}^k)$. Define
\begin{equation}\label{pi}
    \pi^k:=([\lambda^{k+1}]_1,\cdots,[\lambda^{k+1}]_{m_1},[\upsilon^k]_1,\cdots,[\upsilon^k]_{m_2})^\intercal.
\end{equation}
We then deduce from (\ref{finasub})
\begin{equation*}
    H(\textbf{x}^k)^\intercal\pi^k=\delta^k-\theta^k.
\end{equation*}
Likewise, the matrix $H(\textbf{x}^k)H(\textbf{x}^k)^\intercal$ is nonsingular for sufficiently large $k\in\mathcal{K}_1$, and
\begin{equation*}
   \pi^k=[H(\textbf{x}^k)H(\textbf{x}^k)^\intercal]^{-1}H(\textbf{x}^k)(\delta^k-\theta^k).
\end{equation*}
Taking limitations within $\mathcal{K}_1$ on both sides of the expression above for $\pi^k$, we have then
\begin{equation*}
    \pi^k\rightarrow\pi^*=-[H(\textbf{x}^*)H(\textbf{x}^*)^\intercal]^{-1}H(\textbf{x}^*)\theta^*.
\end{equation*}
Taking limitations for $k\in\mathcal{K}_1$ again on both sides of (\ref{finasub}), it follows from (\ref{pi}) that
\begin{equation*}
    \theta^*+\sum_{i=1}^{m_1}[\Lambda^*]_i\nabla[h_1(\textbf{x}^*)]_i+\sum_{i=1}^{m_2}[\upsilon^*]_i\nabla[h_2(\textbf{x}^*)]_i=0,
\end{equation*}
where $\Lambda^*$ and $\upsilon^*$ are guaranteed by $\pi^*$. Therefore, $\textbf{x}^*$ is a $\text{KKT}$ point of problem (\ref{vecpro}).
 \end{proof}

By Theorem \ref{centraltheo} and (\ref{defx}), we can immediately obtain Theorem \ref{congver}. Our numerical experiments in the next section testify that Algorithm \ref{algo1} works well and can output $\text{KKT}$ point of the original problem (\ref{pro2}).

\section{Experiment Study}
In this section, we conduct numerical experiments to show effectiveness of Algorithm \ref{algo1} by using MATLAB (2020a) on a laptop of 16G of memory and Inter Core i7 2.3Ghz CPU against several state-of-the-art unsupervised feature selection methods on six real-world datasets, including one speech signal dataset (Isolet\footnote{https://jundongl.github.io/scikit-feature/datasets.html\label{web}}), two face image datasets (ORL\textsuperscript{\ref {web}},COIL20\textsuperscript{\ref {web}}), three microarray datasets (lung\textsuperscript{\ref {web}}, TOX-171\textsuperscript{\ref {web}}, 9\_Tumors\footnote{https://github.com/primekangkang/Genedata}). Table \ref{table1} summarizes the details of these 6 benchmark datasets used in the experiments. In addition to verifying the effectiveness of our method on the above datasets, we
also show the stability analysis, robustness analysis and parameter sensitivity analysis on some datasets.
\begin{table}[htp]
\caption{Dataset Description}
\vspace{2mm}
\begin{center}
\setlength{\tabcolsep}{1mm}{
\label{table1}
\begin{tabular}{|c|c|c|c|}
\hline
Dataset & Size & \# of Features&\# of Classes\\
\hline
\hline
lung&203&3312&5\\
TOX-171&171&5748&4\\
9\_Tumors&60&5726&9\\
Isolet&1560&617&26\\
ORL&400&1024&40\\
COIL20&1440&1024&20\\
\hline
\end{tabular}
}
\end{center}
\end{table}

{\bf Methods to Compare.} We compare the performance of Algorithm \ref{algo1} with the following state-of-the-art unsupervised feature selection methods:
\begin{itemize}
    \item \textbf{Baseline}: All of the original features are adopted.
    \item \textbf{MaxVar}~\citep{krzanowski1987selection}: Features corresponding to the maximum variance are selected to obtain the expressive features.
    \item \textbf{LS}~\citep{he2005laplacian}: Laplacian Score, in which features are selected with the most consistency with Gaussian Laplacian matrix.
    \item \textbf{SPEC}~\citep{zhao2007spectral}: According to spectrum of the graph to select features.
    \item \textbf{MCFS}~\citep{cai2010unsupervised}: Multi-cluster feature selection, it uses the $l_1$-norm
to regularize the feature selection process as a spectral information regression problem.
    \item \textbf{NDFS}~\citep{li2012unsupervised}:Non-negative discriminative feature selection, which addressed feature discriminability and correlation simultaneously.
    \item \textbf{UDFS}~\citep{yang2011l2}:Unsupervised discriminative feature selection
incorporated discriminative analysis as well as $l_{2,1}$-norm minimization, which is formalized as a unified framework.
\item\textbf{UDPFS}~\citep{wang2020unsupervised}: Unsupervised
discriminative projection for feature selection to select discriminative features by conducting fuzziness learning and sparse
learning simultaneously.
\end{itemize}

{\bf Evaluation Measures.}
Similar to previous work, and basing on the attained clustering results and the ground truth information, we evaluate the performance of the unsupervised feature selection methods by two widely utilized evaluation metrics, i.e., clustering ACCuracy (ACC) and Normalized Mutual Information (NMI)~\citep{yang2011l2}. The higher the ACC and NMI are, the better the clustering performance is.

Given one sample $\textbf{x}_i\in\{\textbf{x}_i\}_{i=1}^n$, denote $y_i$ be the ground truth label and $l_i$ be the predicted clustering label. The ACC is defined as
\begin{equation*}
    \text{ACC}=\frac{1}{n}\sum_{i=1}^n\delta(y_i,map(l_i)),
\end{equation*}
where $\delta(a,b)= 1$ if $a=b$; otherwise $\delta(a,b)=0$, and $map(l_i)$ is the permutation mapping function that maps each cluster label $l_i$ to the equivalent label from the data set. 

Given two random variables $P$ and $Q$, $P$ denotes the true labels and $Q$ represents clustering results. The NMI of $P$ and $Q$ is defined as:
\begin{equation*}
  \text{NMI}(P, Q)= \frac{I(P; Q)}{\sqrt{H(P)H(Q)}},
\end{equation*}
where $I(P; Q)$ is the mutual information between $P$ and $Q$, $H(P)$ and $H(Q)$ are the entropies of $P$ and $Q$, respectively.

\vspace{2mm}
{\bf Experiment Setting.}
In our experiments, the parameters of Algorithm \ref{algo1} are set as follows: $$\tau=0.99,\quad  r=1.01, \quad \rho^1=c/2,  \quad \overline{\lambda}_{1}^1=\overline{\lambda}_{3}^1=\overline{\lambda}_{4}^1=\textbf{O}_{n\times c}, \quad \overline{\lambda}_{2}^1= \textbf{O}_{d\times c}, $$
and $$\overline{\lambda}_{N,min}=-100\textbf{E}, ~~\overline{\lambda}_{N,max}=100\textbf{E}~~~(N=1, 2, 3, 4), \quad
  \epsilon^k=0.995^k~~(k\in\mathbb{N}).
$$
The parameters in Algorithm \ref{algo2} are set as $\underline{C}=C_i^{k,j}=\overline{C}=0.5$. The iteration is terminated if the iteration number exceeds 20.

In the compared methods, there are some hyper-parameters to be set in advance. We fix number of neighboring parameter $k=5$ for LS, SPEC, MCFS, UDFS, NDFS, and our proposed method. In order to make fair comparison of different unsupervised feature selection methods, we tuned the parameters for all methods by a grid-search strategy from $\{10^{-6},10^{-5},10^{-4}, \cdots, 10^4, 10^5, 10^6\}$, and the best clustering results from the optimal parameters are reported for all the algorithms. Because the optimal number of selected features is unknown, we set different number of selected features for all datasets, the selected feature number was tuned from $\{50, 100, 150, 200, 250, 300\}$. After completing the feature selection process, we use $K$-means algorithm to cluster the data into $c$ groups.
Since the initial center points have great impact on the performance of $K$-means algorithm, we conduct $K$-means algorithm $20$ times repeatedly with random initialization to report the mean and standard deviation values of ACC and NMI.

\vspace{2mm}In the next subsections, we will illustrate the algorithmic performance, stability, robustness and parameter sensitivity, respectively.

\subsection{Algorithmic Performance}
The experiments results of different methods on the datasets are summarized in Tables \ref{tab2} and \ref{tab3}. The best results are highlighted in bold fonts. 

In view of the averaging of all numerical results, it can be seen that the performance of our method is superior to other state-of-the-art methods. Its good performance is mainly attributed to the following aspects: Firstly, we adopt the technology similar to NDFS to establish the model, i.e., learning the pseudo class label indicators and the feature selection matrix simultaneously. However, the difference is that we use $l_{2,1}$-norm to characterize the linear loss function between features and pseudo labels and also take into account the prevention of overfitting. Secondly, different from the commonly used processing methods, we apply a convergent algorithm that can simultaneously optimize all variables in the feature selection model. In the previous section, we have proven the convergence property of our algorithm. Since the iterative sequence of our algorithm converges to $\text{KKT}$ points, it achieves better results than other methods.
\begin{table*}[htp]
	\centering
	\fontsize{9}{12}\selectfont
	\caption{Clustering results (ACC$\pm$STD$\%$) of different feature selection algorithms on six real-world datasets. The best results are highlighted in bold.}
	\label{tab2}
	\resizebox{\textwidth}{!}{
	\begin{tabular}{|c|c|c|c|c|c|c|c|c|c|}
	\hline
	\hline
	Dataset &All features& LS & Maxvar& MCFS & NDFS & SPEC & UDFS & UDPFS & Ours\cr\hline
	lung &65.0$\pm$3.6& 74.9$\pm$0.2& 68.0$\pm$9.4& 77.6$\pm$11.0 &63.3$\pm$6.9& 64.1$\pm$7.9& 72.3$\pm$10.9 & 69.6$\pm$7.7 & \textbf{82.4$\pm$7.9}\cr\hline
	ORL &49.7$\pm$3.2& 49.9$\pm$2.4& 50.8$\pm$1.4& \textbf{55.7$\pm$3.7} &50.5$\pm$3.0& 51.4$\pm$ 2.2& 53.3$\pm$4.1 &  53.1$\pm$3.8& 52.9$\pm$3.4\cr\hline
	Isolet& 60.9$\pm$2.1&58.7$\pm$1.5&56.9$\pm$2.3&64.5$\pm$4.3& 61.6$\pm$4.4&56.5$\pm$3.0&57.8$\pm$ 3.1& 58.3$\pm$2.9 &\textbf{65.8$\pm$3.9}\cr\hline
	COIL20& 62.7$\pm$3.1&62.2$\pm$1.9&61.4$\pm$1.6&63.0$\pm$3.7& 58.7$\pm$ 4.1&\textbf{65.5$\pm$3.8}&60.2$\pm$4.2& 58.2 $\pm$4.6&61.6$\pm$3.8\cr\hline
	TOX-171&42.8$\pm$2.1&43.1$\pm$1.4&42.9$\pm$1.6&42.9$\pm$1.6&43.4$\pm$3.3 &40.4$\pm$0.0&48.2$\pm$2.1& \textbf{54.0$\pm$ 3.2}&49.2$\pm$4.1  \cr\hline
	9\_Tumors&40.8$\pm$3.7&42.3$\pm$2.6&41.2$\pm$2.6&42.4$\pm$3.6&44.0$\pm$3.7 &35.8$\pm$2.4&43.0$\pm$ 4.3 & \textbf{44.2$\pm$4.3}&44.1$\pm$4.1 \cr\hline
	\textbf{Mean} &53.7$\pm$3.0&55.2$\pm$1.7&53.5$\pm$3.2&57.7$\pm$4.7&53.6$\pm$4.2 &52.3$\pm$3.2&55.8$\pm$4.8 &56.2$\pm$4.4 & \textbf{59.3$\pm$4.5} \cr\hline
	\end{tabular}%
    }
\end{table*}%
\begin{table*}[htp]
	\centering
	\fontsize{9}{12}\selectfont
	\caption{Clustering results (NMI$\pm$STD$\%$) of different feature selection algorithms on six real-world datasets. The best results are highlighted in bold.}
	\label{tab3}
	\resizebox{\textwidth}{!}{
	\begin{tabular}{|c|c|c|c|c|c|c|c|c|c|}
	\hline
	\hline
	Dataset &All features& LS & Maxvar& MCFS & NDFS & SPEC & UDFS & UDPFS & Ours\cr\hline
	lung &51.6$\pm$1.9& 53.1$\pm$ 0.5& 57.8$\pm$ 3.9& 67.5$\pm$7.0 &53.0$\pm$3.5 & 52.5$\pm$ 5.6& 61.3$\pm$5.8& 59.0$\pm$4.0&\textbf{69.0$\pm$4.4}\cr\hline
		ORL &70.0$\pm$1.7& 71.1$\pm$ 1.3& 70.7$\pm$ 2.1& \textbf{76.8$\pm$1.8} &73.2$\pm$1.9 & 71.4$\pm$ 1.3& 74.7$\pm$1.6& 74.8$\pm$1.6 & 74.9$\pm$1.7\cr\hline
	Isolet& 75.7$\pm$0.8&73.2$\pm$0.9&74.8$\pm$1.3&77.7$\pm$1.7& 77.1$\pm$ 2.2&72.4$\pm$1.1&74.7$\pm$1.8&74.4$\pm$1.3 &\textbf{80.5$\pm$1.3}\cr\hline
	COIL20& \textbf{77.1$\pm$1.3}&72.5$\pm$1.1&71.9$\pm$0.7&76.5$\pm$1.7& 74.0$\pm$ 1.6&75.3$\pm$1.6&75.4$\pm$1.3&73.9$\pm$2.0 &76.3$\pm$2.3\cr\hline
	TOX-171&13.6$\pm$2.3&12.5$\pm$1.7&11.4$\pm$3.2&12.7$\pm$0.4&16.4 $\pm$5.9 & 9.7 $\pm$ 0.0 &22.8$\pm$3.5&\textbf{29.9$\pm$1.2} &25.3$\pm$4.4  \cr\hline
    9\_Tumors&39.5$\pm$3.1&41.0$\pm$2.3&40.2$\pm$2.5&41.1$\pm$2.7&44.7$\pm$4.5  &34.5$\pm$2.4& 44.1$\pm$4.3 & \textbf{46.7$\pm$3.6}&44.8$\pm$3.2 \cr\hline
    \textbf{Mean}&54.6$\pm$1.9&53.9$\pm$1.3&54.5$\pm$2.3&58.7$\pm$2.6&56.4$\pm$3.3 &52.6$\pm$2& 58.8$\pm$3.1 & 59.8$\pm$2.3 &\textbf{61.8$\pm$2.9} \cr\hline
	\end{tabular}
	}
\end{table*}%

\subsection{Stability Analysis}
Now we will illustrate that our algorithm is more stable than other iterative algorithms including: UDPFS~\citep{wang2020unsupervised}, NDFS~\citep{li2012unsupervised} and UDFS~\citep{yang2011l2}. Following the symbol in ~\citet{li2012unsupervised,yang2011l2,wang2020unsupervised}, we denote the feature selection matrix as $W$ in these methods and define $$\eta=\frac{\Vert W_{k+1}-W_{k}\Vert_F}{\Vert W_{k}-W_{k-1}\Vert_F},$$ where $W_k$ is the $k$-th iterative point. To demonstrate fully that our algorithm is more stable,  we randomly initialize  cluster indicator matrix $Y$ and $W$ 20 times. Under the parameter setting of the optimal results obtained by corresponding method, we record the average results of $\eta$. The experimental results are shown in Fig. \ref{fig:3}.
\begin{figure}[t]
\centering
\subfigure[lung]{
 \includegraphics[width=3in]{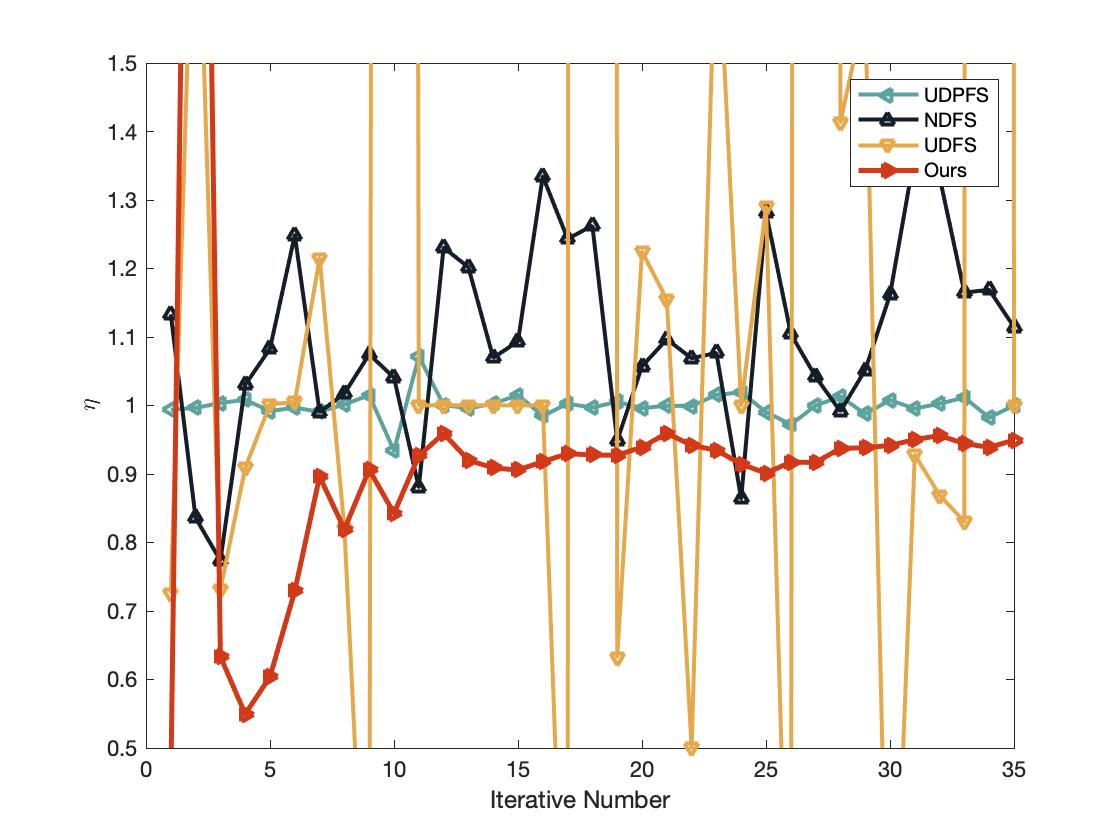}
}\enspace
\subfigure[Isolet]{
 \includegraphics[width=3in]{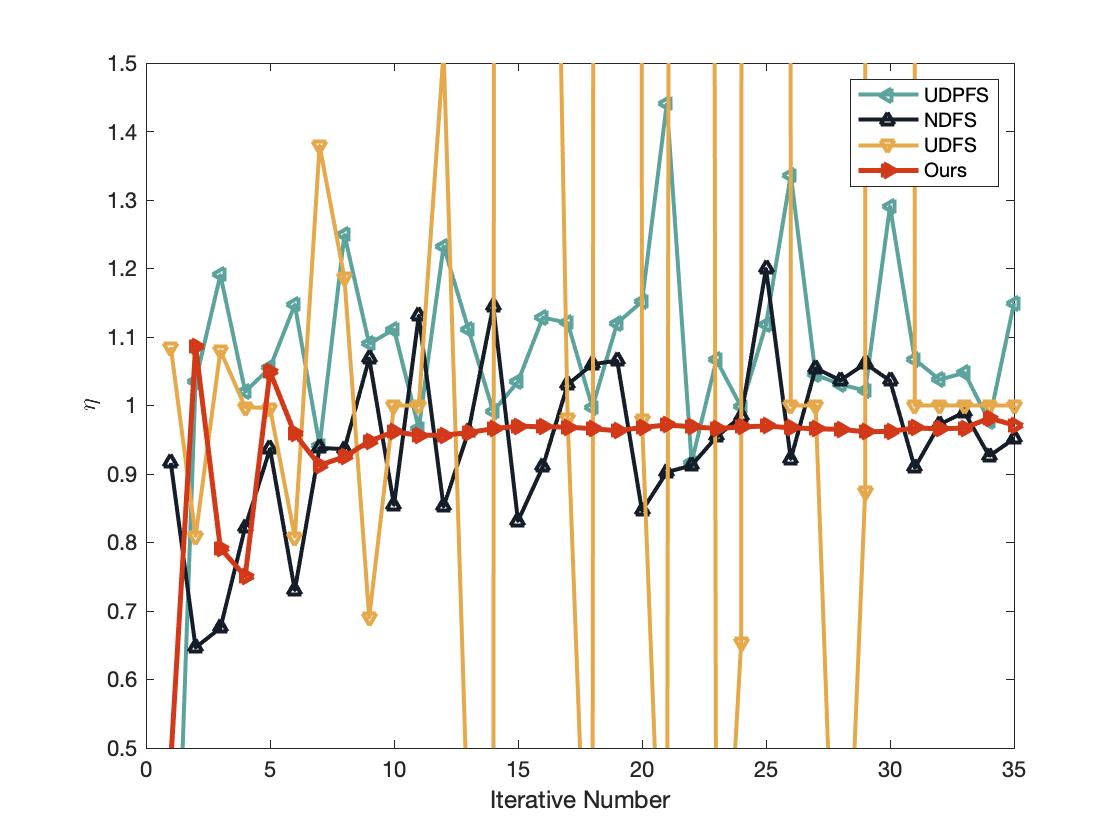}
}

\caption{Stability curves over lung and Isolet.\label{fig:3}}
\end{figure}

It can be seen that the value of $\eta$ of the other three methods are always changing irregularly , while ours starts to stabilize after fewer iterations and then always less than 1. Furthermore, we know that , with the increase of iterative number $k$, $\Vert W_{k+1}-W_{k}\Vert_F$ decreases gradually in our method, which shows that our iterative sequence $\{W_k\}_{k\in\mathbb{N}}$ keeps the "distance" of the adjacent two points gradually reduced and it is changed regularly according to the iterative rules.  Following the previous theoretical proof, iterative sequence $\{W_k\}_{k\in\mathbb{N}}$ will eventually converge to the $\text{KKT}$ points. Compared with our method, since the values of $\eta$ of UDPFS, NDFS and UDFS
are ruleless, iterative sequence $\{W_k\}_{k\in\mathbb{N}}$ is ``jumping" irregularly and does not have a convergence trend. Therefore, our method is more stable.

\subsection{Robustness Analysis}
\begin{figure}[t]
\centering
\subfigure[ ]{
 \includegraphics[width=2.9in]{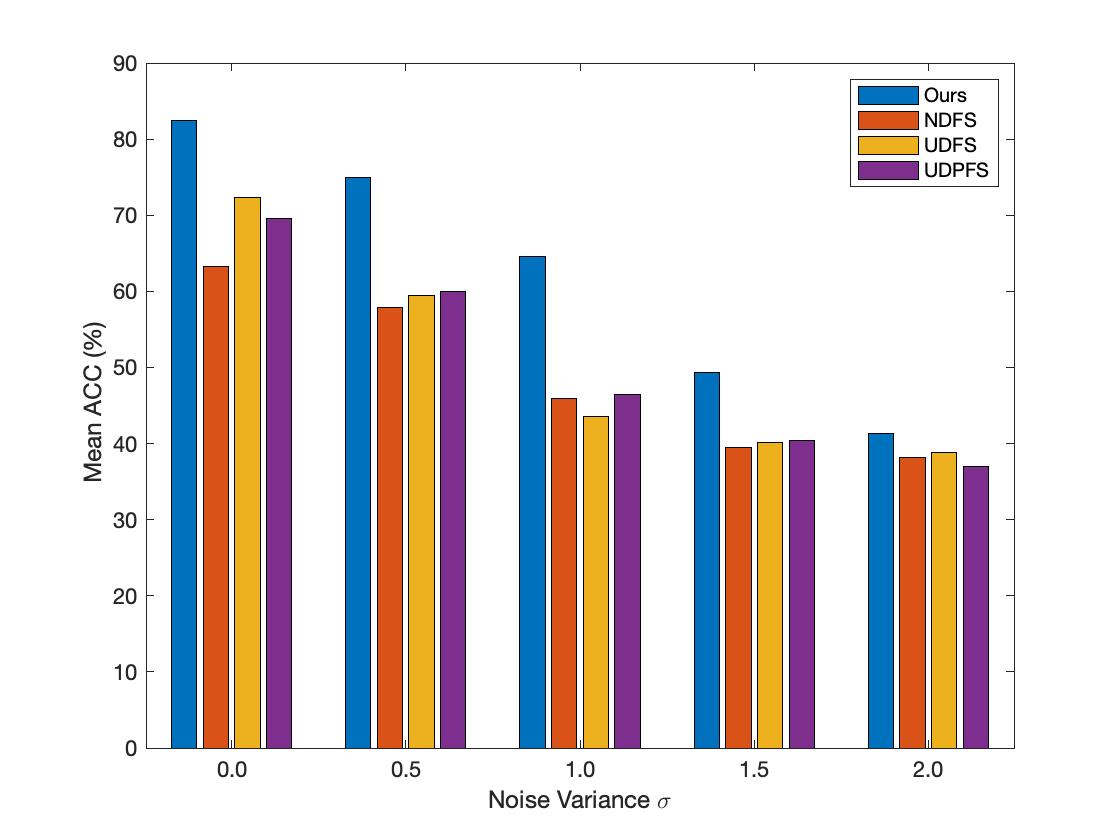}
}\enspace
\subfigure[ ]{
 \includegraphics[width=2.9in]{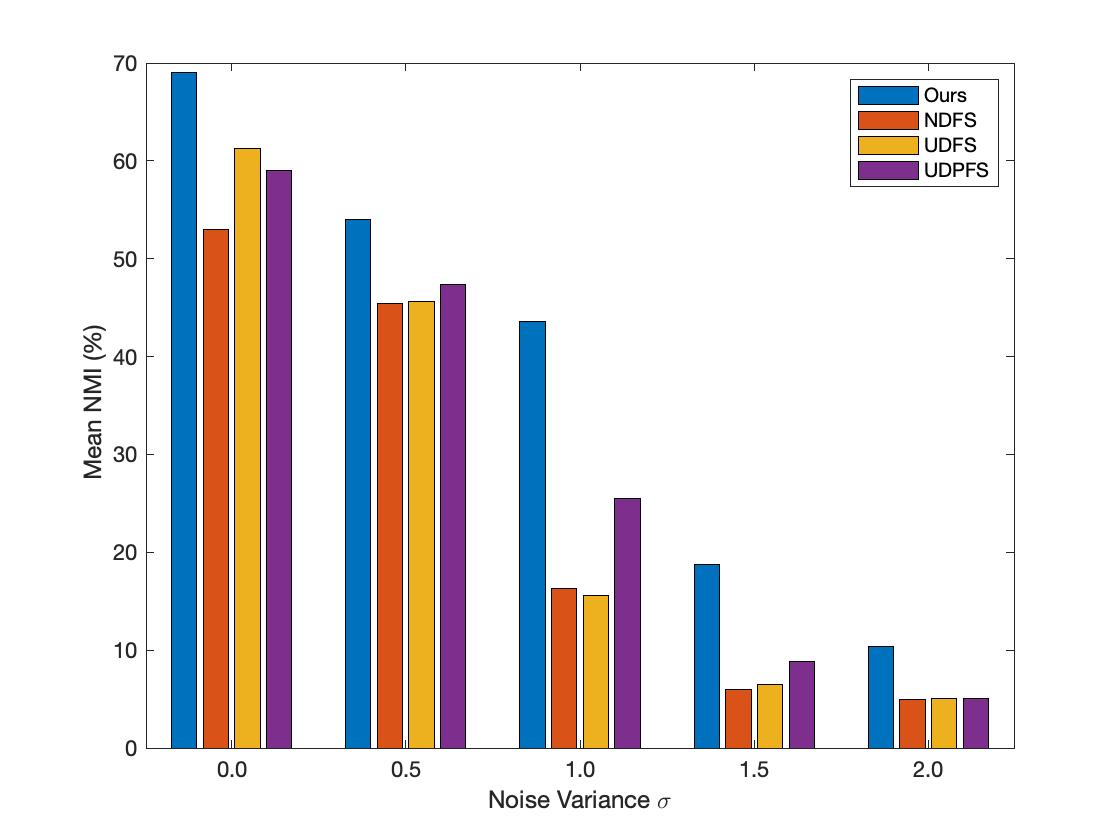}
}
 \vspace{-3mm}
\caption{Robustness comparison to data perturbation between our method and other iterative methods on lung.\label{fig:4}}
\end{figure}
\begin{figure}[t]
\centering
\subfigure[ ]{
 \includegraphics[width=2.9in]{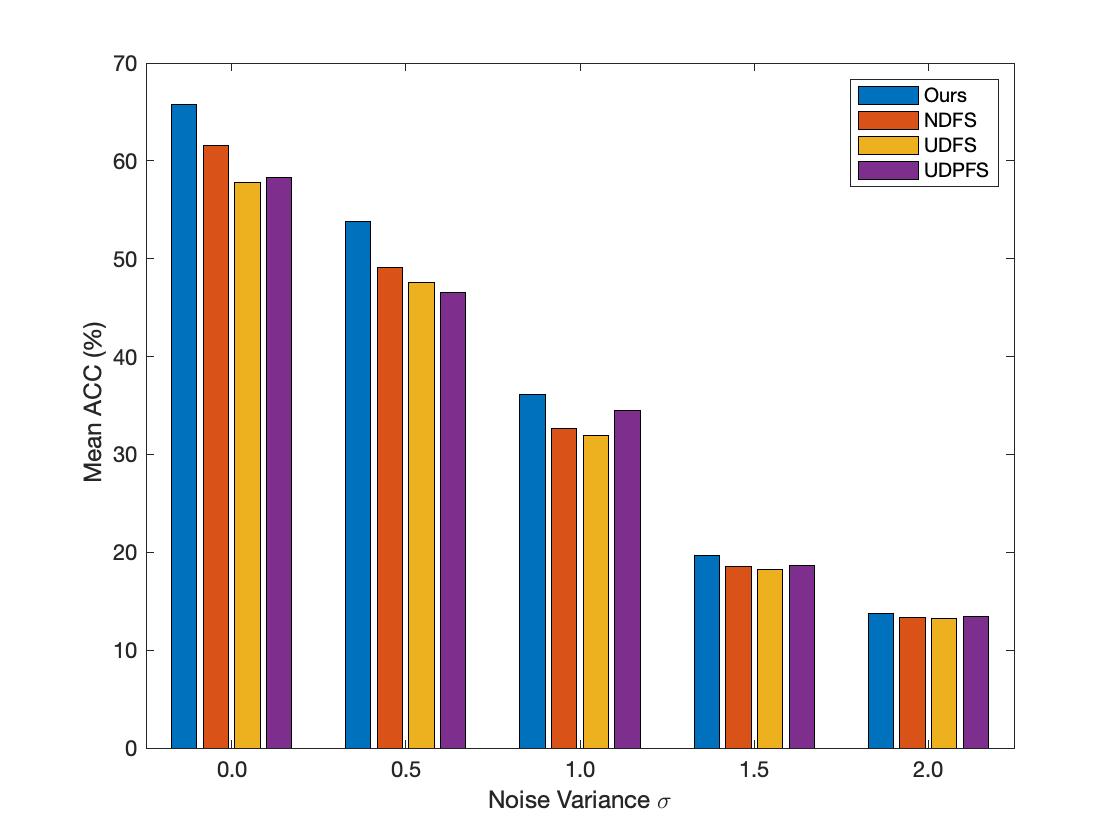}
}\enspace
\subfigure[ ]{
 \includegraphics[width=2.9in]{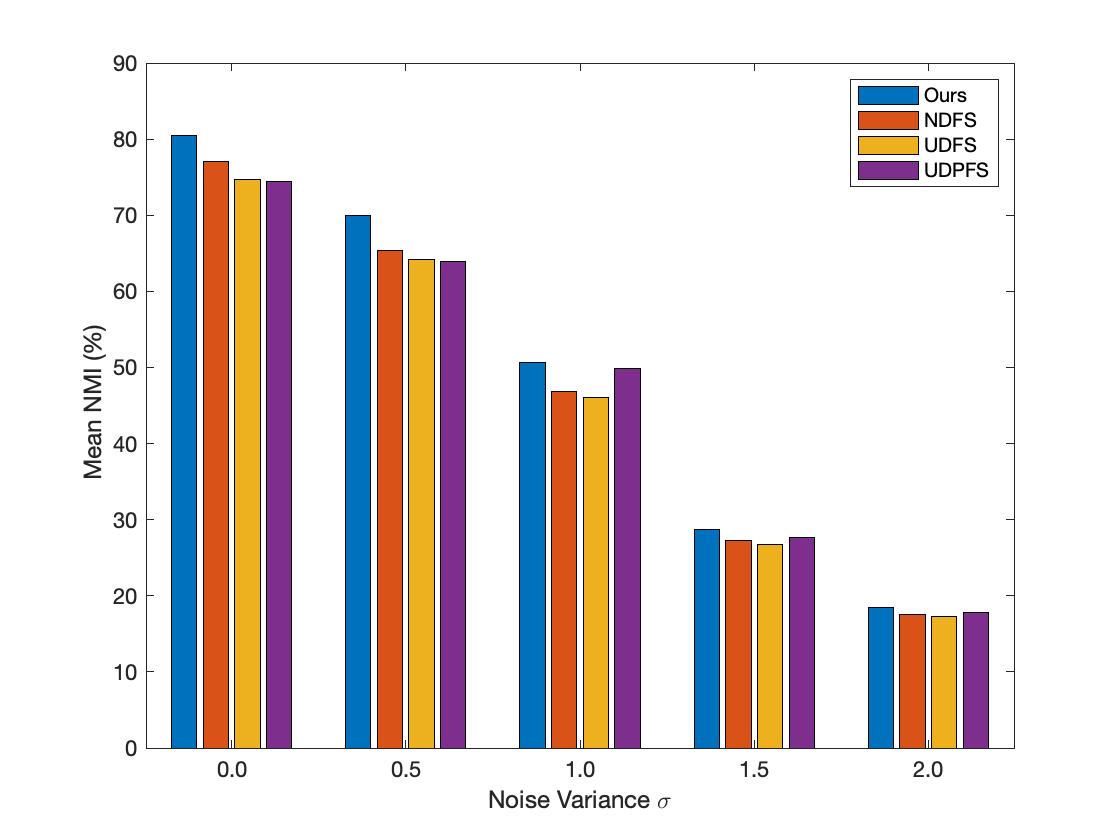}
}
 \vspace{-3mm}
\caption{Robustness comparison to data perturbation between our method and other iterative methods on Isolet.\label{fig:5}}
\end{figure}
\FloatBarrier
In this subsection, we summarize the main results for our robustness analysis. We consider the effect of varying the amount of perturbation introduced in the datasets, i.e., the effect of performance if we fine-tune the Gaussian noise from the distribution $\mathcal{N}(\textbf{0},\sigma^2)$ where $\sigma$ is sampled from the set $\{0.0, 0.5, 1.0, 1.5, 2.0\}$ and add the Gaussian noise to the input data. In order to make a fair comparison, we conduct the experiments under the parameter setting of the optimal results obtained by each method for the chosen dataset.
Meanwhile, in order to avoid the influence of the randomness of noise in a single experiment, we uniformly do ten experiments for each noise variance, and then average the results as the final result. 

Fig. \ref{fig:4} and Fig. \ref{fig:5} show the robustness of the iterative methods here considered on the lung dataset and Isolet dataset with different levels of noise. Note that with the increase of disturbance, the robustness of all iterative methods
falls off, while the performance of our method is always the best. Therefore, compared with other methods, our method has a strong robustness.

\subsection{Parameter Sensitivity Analysis}
\begin{figure}[t]
\centering
\subfigure[ACC over Isolet]{
 \includegraphics[width=2.9in]{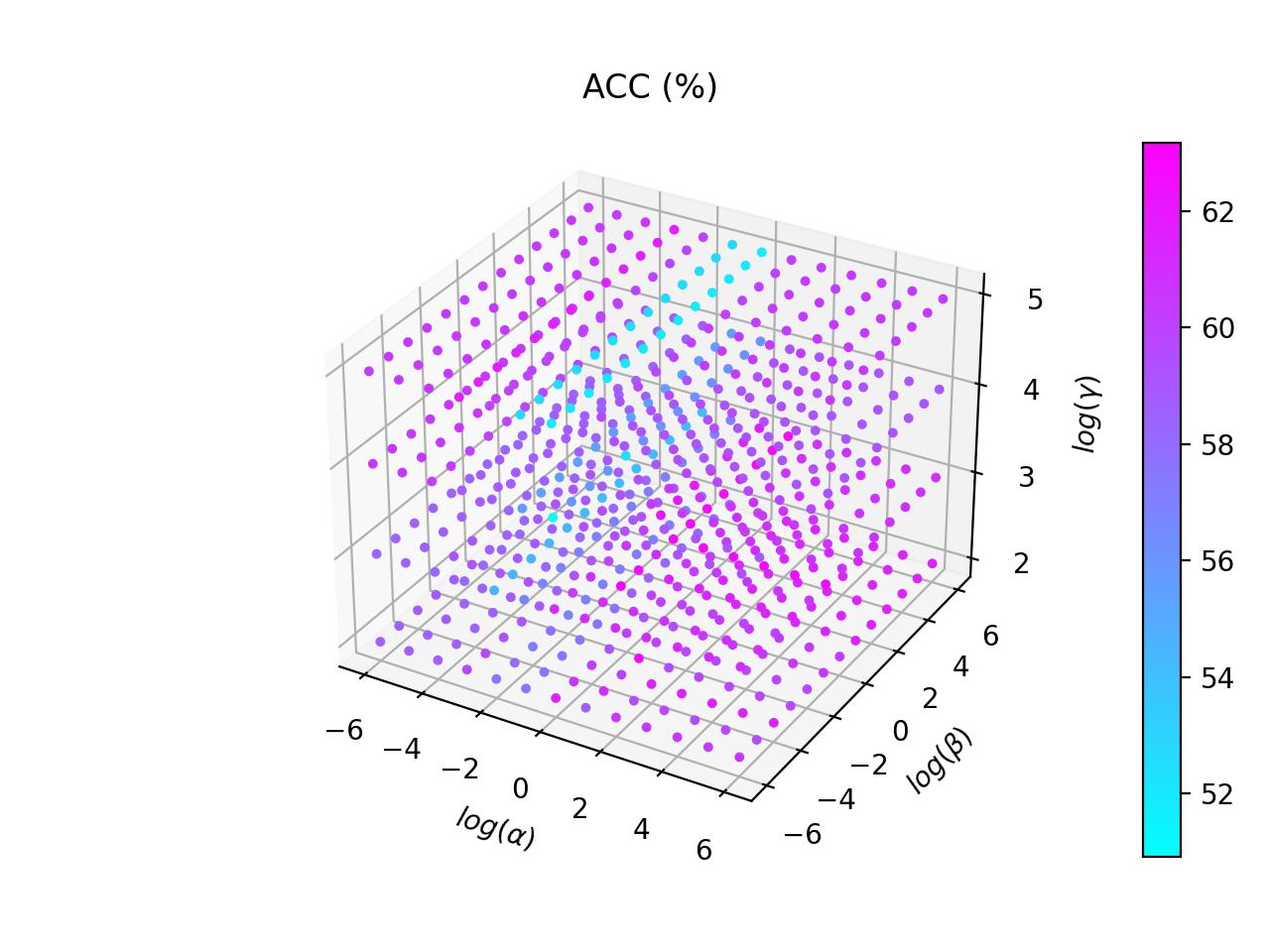}
}\enspace
\subfigure[NMI over Isolet ]{
 \includegraphics[width=2.9in]{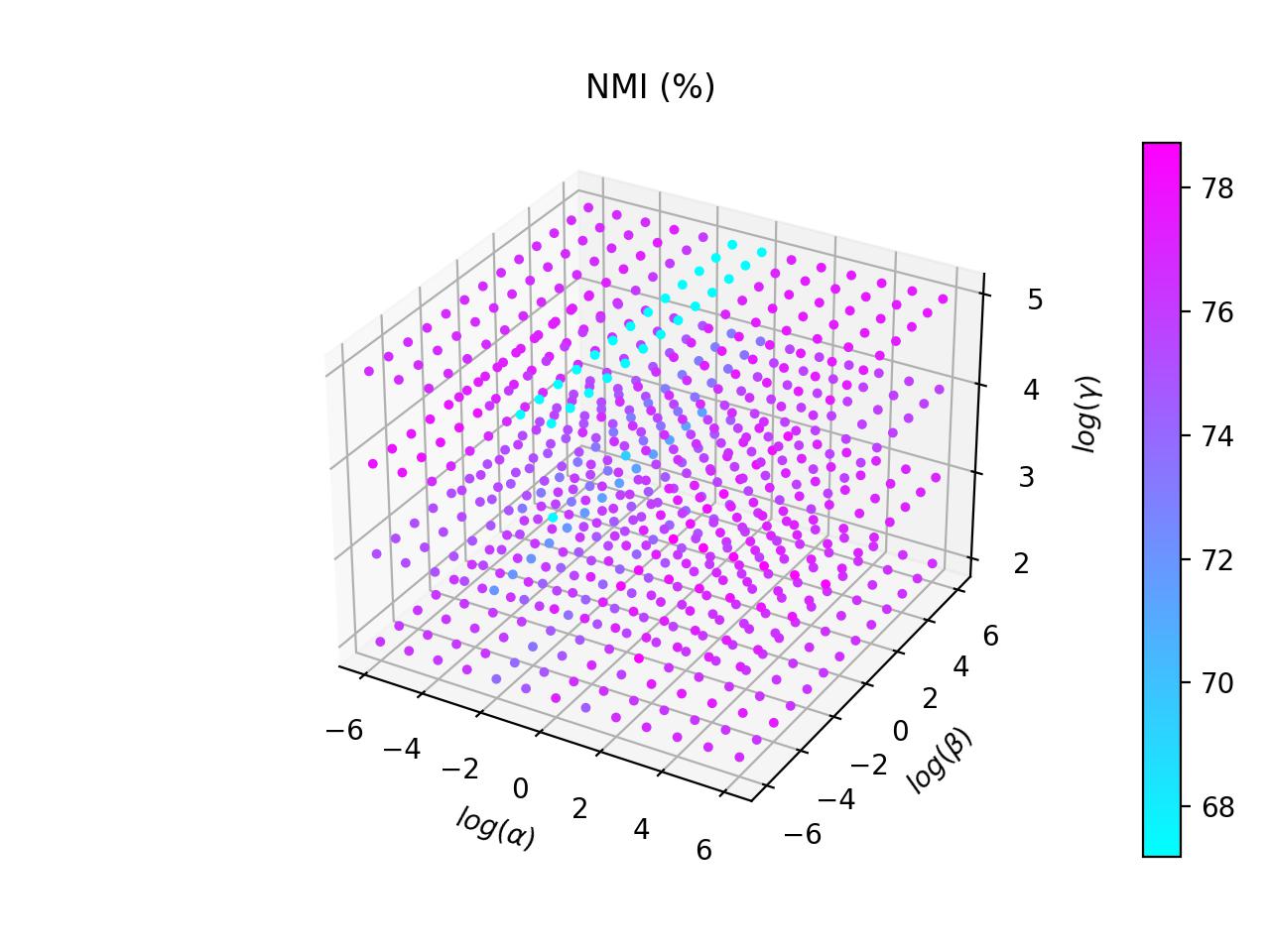}
}
 \centering
\subfigure[ACC over lung ]{
 \includegraphics[width=2.9in]{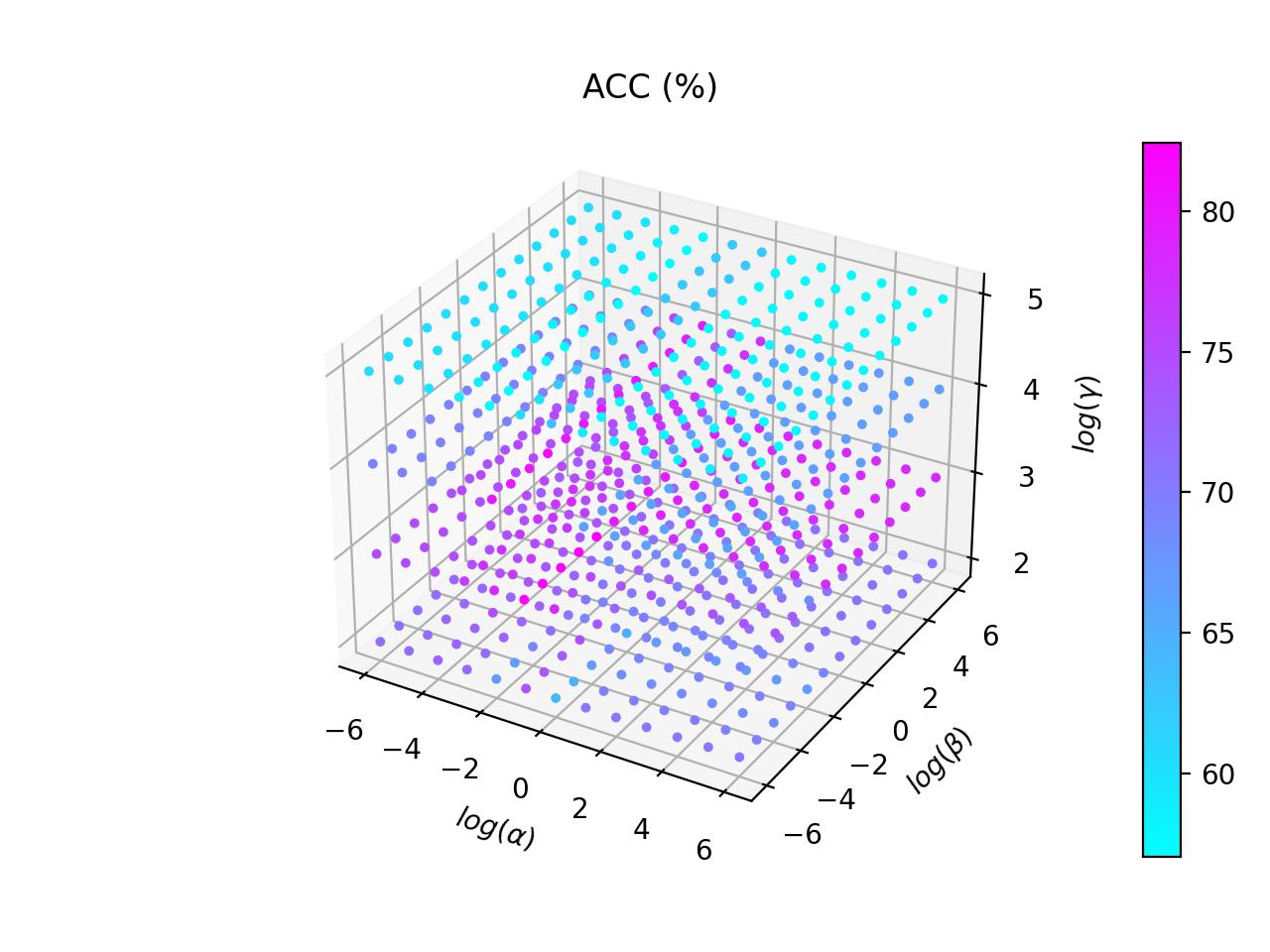}
}\enspace
\subfigure[NMI over lung ]{
 \includegraphics[width=2.9in]{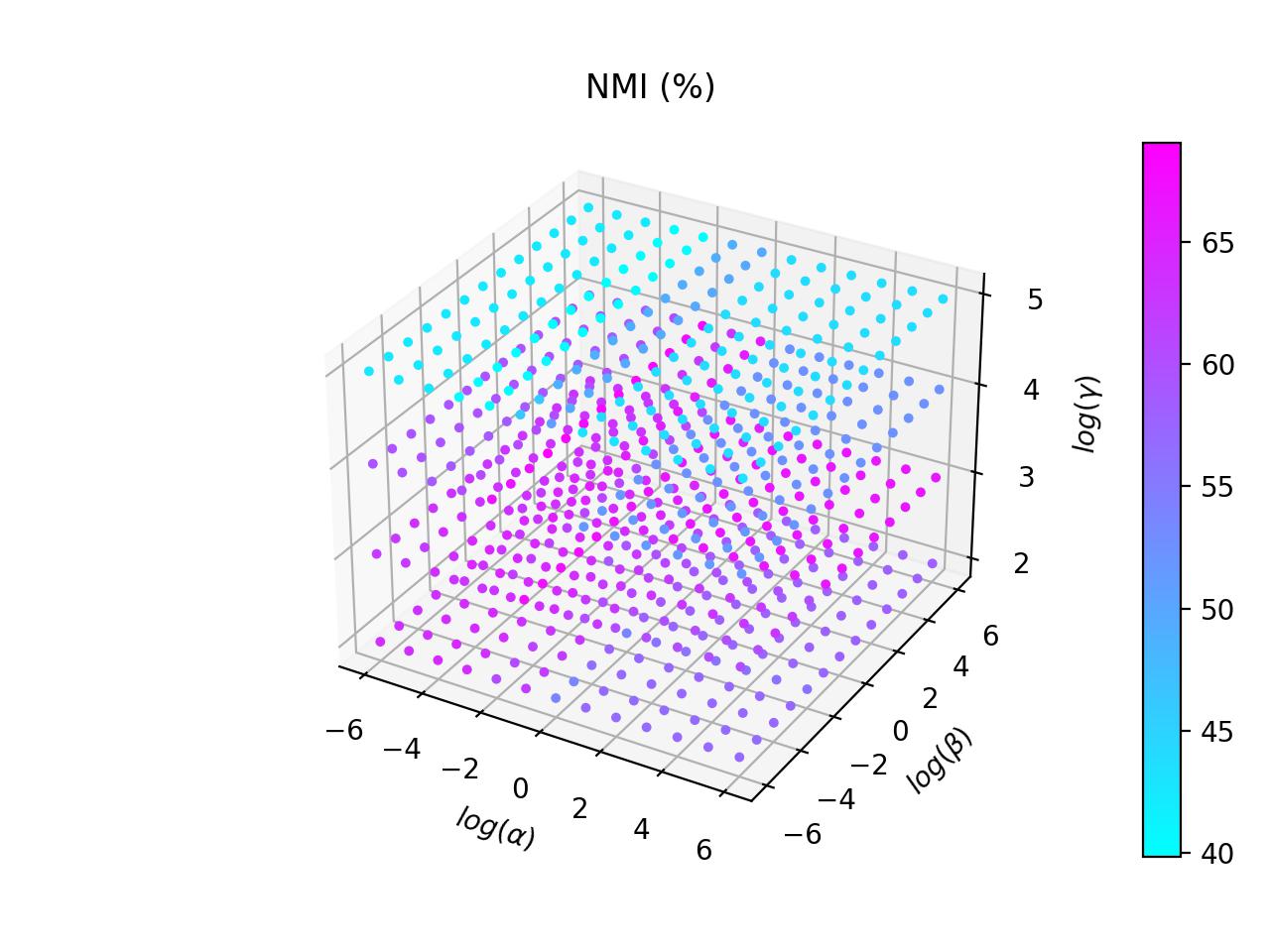}
}
\vspace{-2mm}
\caption{Performance with different $\alpha$, $\beta$, $\gamma$ values on Isolet and lung with a grid search strategy.\label{fig:1+2}}
\end{figure}
Like many other feature selection algorithms, our proposed method also requires several parameters $\alpha, \beta, \gamma$ to be set in advance. Next, we will discuss their sensitivity. In our experiments, we observe that the parameters $\alpha$ and $\beta$ have more effect on the performance than the parameter $\gamma$ on the given datasets. Therefore, we focus on discussing the parameters $\alpha$ and $\beta$. We will conduct the parameter sensitivity study in terms of $\alpha$, $\beta$ when $\gamma$ is fixed to some values. $\alpha$ and $\beta$ are tuned from $\{10^{-6}, 10^{-5},\cdots,10^5, 10^6 \}$. The results on lung and Isolet are presented in Fig. \ref{fig:1+2}. It can be seen that our method is not sensitive to $\alpha, \beta$ and $\gamma$ with relatively wide ranges.

\section{conclusion}
In this paper, we firstly have explored an ideal feature selection model: $l_{2,1}$-norm regularized regression optimization problem with non-negative orthogonal constraint, which well captures the most representative features from the original high-dimensional data. Then, we propose an inexact augmented Lagrangian multiplier method to solve our feature selection model. Moreover, a proximal alternating minimization method is utilized to solve the augmented Lagrangian subproblem with the benefit being that each subproblem has a closed form solution.  It is shown that our algorithm  has the subsequence convergence property, which is not provided in the state-of-the-art unsupervised feature selection methods. Quantitative and qualitative experimental results have shown the effectiveness of our proposed method.

\addcontentsline{toc}{section}{References}
\bibliography{main}

\end{document}